\documentclass[11pt, oneside]{amsart}
\usepackage{amscd,amsmath,amssymb,amsfonts,mathrsfs, mathtools, float, verbatim,fullpage}

\usepackage[
    backend=biber,
    style=alphabetic, 
    backref=true,
    maxalphanames=99, 
    minalphanames=1,  
    maxbibnames=99,   
    url=false,
    doi=false,      
    eprint=false,   
    isbn=false      
]{biblatex}
\DeclareFieldFormat
  [article,incollection,inproceedings,unpublished]
  {title}{\mkbibemph{#1}}

\DeclareFieldFormat
  [book,thesis,report,misc]
  {title}{\mkbibemph{#1}}

\renewbibmacro*{pageref}{
  \iflistundef{pageref}
    {}
    {\printlist[pageref][-\value{listtotal}]{pageref}}}

\renewbibmacro{in:}{}
\DeclareDelimFormat{nametitledelim}{\addcomma\space}

\allowdisplaybreaks
\usepackage{etoolbox}

\usepackage{setspace}

\makeatletter
\patchcmd{\@setaddresses}{\indent}{\noindent}{}{}
\patchcmd{\@setaddresses}{\indent}{\noindent}{}{}
\patchcmd{\@setaddresses}{\indent}{\noindent}{}{}
\patchcmd{\@setaddresses}{\indent}{\noindent}{}{}
\makeatother

\usepackage[T1]{fontenc}
\usepackage[T1]{eulervm}
\usepackage[cmtip, all]{xy}
\usepackage{tikz-cd}
\usepackage{stackrel}
\usepackage[letterpaper, left=2.4cm, right=2.4cm, top = 2.5cm, bottom = 2.5cm ]{geometry}
\usepackage{tgcursor}
\usepackage{hyperref}
\usepackage{xcolor}
\hypersetup{
    colorlinks,
    linkcolor={blue!80!black},
    citecolor={red!60!black},
    urlcolor={magenta!95!black}
}

\newtheorem{thm}{Theorem}[section]
\newtheorem{prop}[thm]{Proposition}
\newtheorem{lem}[thm]{Lemma}

\newtheorem{cor}[thm]{Corollary}

\newtheorem{assumption}[thm]{Assumption}

\theoremstyle{definition}

\newtheorem{defn}[thm]{Definition}
\theoremstyle{remark}
\newtheorem{remk}[thm]{Remark}
\newtheorem{remks}[thm]{Remarks}

\newtheorem{exm}[thm]{Example}
\newtheorem{exms}[thm]{Examples}
\newtheorem{notat}[thm]{Notation}
\numberwithin{equation}{section}

{\hfill$\square$\end{defn}}
{\hfill$\square$\end{remk}}
{\hfill$\square$\end{remks}}
{\hfill$\square$\end{exm}}
{\hfill$\square$\end{exms}}
{\hfill$\square$\end{notat}}

\newcommand{\thmref}{Theorem~\ref}

\newcommand{\defref}{Definition~\ref}
\newcommand{\lemref}{Lemma~\ref}

\newcommand{\remref}{Remark~\ref}

\newcommand{\secref}{Section~\ref}

\newcommand{\sC}{{\mathcal C}}
\newcommand{\sD}{{\mathcal D}}

\newcommand{\sF}{{\mathcal F}}
\newcommand{\sG}{{\mathcal G}}

\newcommand{\sN}{{\mathcal N}}
\newcommand{\sO}{{\mathcal O}}

\newcommand{\sR}{{\mathcal R}}
\newcommand{\sS}{{\mathcal S}}
\newcommand{\sT}{{\mathcal T}}

\newcommand{\sX}{{\mathcal X}}

\newcommand{\A}{{\mathbb A}}
\newcommand{\B}{{\mathbb B}}

\newcommand{\D}{{\mathbb D}}
\newcommand{\E}{{\mathbb E}}

\newcommand{\G}{{\mathbb G}}
\renewcommand{\L}{{\mathbb L}}

\newcommand{\inj}{\hookrightarrow}

\newcommand{\rank}{{\rm rank}}
\newcommand{\Pic}{{\rm Pic}}

\newcommand{\Hom}{{\rm Hom}}

\newcommand{\Spec}{{\rm Spec \,}}

\newcommand{\sHom}{{\mathcal{H}{om}}}

\newcommand{\holim}{\mathop{{\rm holim}}}

\newcommand{\hocolim}{\mathop{{\rm hocolim}}}

\newcommand{\Tot}{{\operatorname{\rm Tot}}}

\newcommand{\Sym}{{\operatorname{\rm Sym}}}

\newcommand{\DM}{{\operatorname{\mathcal{DM}}}}

\newcommand{\End}{{\operatorname{\text{End}}}}

\newcommand{\ds}{{/\kern-3pt/}}

\newcommand{\Qcoh}{Qcoh}

\newcommand{\Map}{\mathcal{M}ap}
\newcommand{\IndCoh}{IndCoh}

\newcommand{\Fun}{{\rm Fun}}

\newcommand{\Ind}{{\operatorname{Ind}}}

\newcommand{\colim}{\mathop{\text{\rm colim}}}

\newcommand{\hocofib}{{\operatorname{\rm hocofib}}}
\newcommand{\hofib}{{\operatorname{\rm hofib}}}

\newcommand{\Mod}{{\mathbf{Mod}}}

\newcommand{\ov}{\overline}

\renewcommand{\dim}{\text{\rm dim}}

\newcommand{\tuborg}{\left\{\begin{array}{ll}}
\newcommand{\sluttuborg}{\end{array}\right.}

\newcommand{\pr}{{\rm pr}}

\newcommand{\wt}{\widetilde}
\newcommand{\wh}{\widehat}

\def\cO{\mathcal{O}}

\def\ol#1{\overline{#1}}

\newcounter{elno}

\newcounter{elno-abc}   

\newcounter{elno-abc-prime}   

\usepackage{epigraph}

\begin{document}
\sloppy
\title[HKR coherent matrix factorization category]
{A Hochschild-Kostant-Rosenberg type theorem for coherent matrix factorizations}
\author{Arnab Kundu$^{1,2}$}  
\author{Bhamidi Sreedhar$^{1,2}$}
 \address{$^{1}$ Harish-Chandra Research Institute, Chhatnag Road, Jhunsi, Prayagraj 211 019, India}
 \address{$^{2}$ Homi Bhabha National Institute, Training School Complex, Anushakti Nagar, Mumbai 400 094, India}
\email{arnabkundu@hri.res.in}
\email{bhamidisreedhar@hri.res.in}

\begin{abstract}
In this article, we prove a Hochschild-Kostant-Rosenberg type theorem for the category of coherent matrix factorizations associated to a Landau-Ginzburg model $(X,f)$, where $X$ is a quasi-smooth derived Deligne-Mumford stack over an algebraically closed field of characteristic 0. 
\end{abstract} 
\setcounter{tocdepth}{1}
\maketitle

\tableofcontents  
{\hypersetup{linkcolor=black} \tableofcontents}

\setlength{\parskip}{12pt} 
\setlength{\abovedisplayskip}{3pt}
\setlength{\belowdisplayskip}{3pt}
\setlength{\abovedisplayshortskip}{0pt}
\setlength{\belowdisplayshortskip}{0pt}

\section{Introduction}
Let $k$ be a field of characteristic $0$. For a smooth affine scheme $X$, the Hochschild-Kostant-Rosenberg (HKR) theorem provides a canonical
isomorphism between the Hochschild homology of the scheme $HH_*(X)$ and the module of differential forms $\Omega^i_X$. That is, one has the following isomorphism, 
\cite{hkr-original},
\[
\Omega^i_X \xrightarrow{\simeq}  \mathrm{HH}_i(X).
\]

The HKR theorem has been generalized in several different directions over the past few decades. Swan proved the HKR theorem for smooth schemes \cite{swan1996hochschild}, which was further generalized by Yekutieli in \cite{yekutieli}. From the perspective of derived algebraic geometry, the work of Ben-Zvi and Nadler \cite{Ben_Zvi_2012} reinterpreted the HKR isomorphism as an equivalence between the derived loop space $\mathcal{L}X:= X \times_{X \times X} X$ and the total space of the shifted tangent complex $T[-1]X$. In recent work \cite{HKR2026},
Fu, Porta, Scherotzke, and Sibilla generalize the work of \cite{Ben_Zvi_2012} and \cite{ArinkinCaldararuHablicsek2019} to non-smooth derived Deligne-Mumford stacks, thereby proving a very general HKR-type theorem. 

A Landau-Ginzburg (LG) model is a pair $(X,f)$ where $X$ is a (derived) Deligne-Mumford stack and $f: X \to \A^1_k$ is a morphism. The study of LG models has been 
motivated by ideas originating from the mathematical theory of mirror symmetry and, in 
recent years, has had several applications in curve counting theories (see 
\cite{FanJarvisRuan2013},\cite{PolishchukVaintrob2016}, \cite{GLSMFundamental2022} 
and \cite{FaveroKim2020}). The category of matrix factorizations is the natural 
category associated with a LG model, playing a role analogous to the derived 
category of (quasi-)coherent sheaves associated with a scheme or stack. One can define the 
Hochschild homology of this category, and it was a natural question to study a HKR-type theorem in this context. In \cite[Section 2]{PV-DUKE}, Polishchuk and Vaintrob 
proved HKR type isomorphisms for LG models of the form  $([X/G],f)$ where $X$ is a smooth 
affine scheme with an action of a finite group $G$. In the same direction, there are 
works of C\u{a}ld\u{a}raru and Tu in \cite[Section 4]{caldararu-tu} and Segal in 
\cite{segal-closed-state} for curved $A_{\infty}$-algebras. There are various 
approaches to proving HKR-type isomorphisms for categories of matrix factorizations. 
In \cite{chain-level-hkr-type-map-chern--kim}, Chung, Kim, and Kim proved a chain-level 
HKR-type isomorphism of mixed complexes between the Hochschild homology complex and 
twisted Hodge complex for an LG model $(X,f)$ where $X$ is assumed to be a  
smooth scheme $X$ or a smooth scheme with finite group action. Halpern-Leistner and Pomerleano \cite{HLP_hodge}, as well as Preygel \cite{Preygel:2011}, have also proved HKR-type theorems for LG models associated to smooth Deligne–Mumford stacks. Using ideas from \cite{HLP_hodge}, a version of this theorem has also been obtained by Choa, Kim, and the second author in \cite{CKS}. 

In this paper, our goal is to generalize the HKR-type theorem to LG models of the form $(X,f)$ where $X$ is not necessarily a smooth  Deligne-Mumford stack. In the smooth setting, a celebrated result of Orlov (\cite[Proposition 3]{Orlov-mf}), proves that for a regular scheme with a flat potential, the category of matrix factorizations is equivalent to the singularity category of the zero locus. In \cite{MR3877165}, Blanc, Robalo, T\"oen, and Vezzosi;\cite{pippi_HS} Pippi; and in \cite{coherent-analogues}, Efimov and Positselski defined a relative singularity category/ category of coherent factorizations associated to an LG model $(X,f)$ which satisfies an Orlov kind of equivalence for non-smooth schemes. For the category of coherent factorizations associated to a quasi-smooth derived Deligne-Mumford stack, we prove a version of the HKR theorem using techniques developed by Preygel in \cite{Preygel:2011} and the recent work of Fu, Porta, Scherotzke, and Sibilla \cite{HKR2026}.

\subsection{Statement of main theorems}
We now state the main theorems of this text. For a detailed explanation of the notations, please see \secref{sec:HKR}.
\begin{thm}[\thmref{thm: HKR homo}]\label{thm:homo-intro}
	Let $X$ be a quasi-smooth, finite type, 
 separated, derived $\DM$ stack over a field $k$ of char 0 with affine stabilizers.
Let $(X,f)$ be a LG model satisfying Assumption \ref{assumption}. 
Then we have the following equivalence of $k((\beta))$-modules:
\begin{equation}\label{hkr isom mf-intro}
    HH_{\bullet}(MF^{\mathrm{coh}}(X,f)) \simeq R\Gamma(I^{DM}X, Sym(\mathbb{L}_{I^{DM}X}[1]))^{B{\mathbb{G}_a}}\otimes_{k[[\beta]]} k((\beta)).
\end{equation}
where $I^{DM}X$ is the orbifold inertia of $X$ and $\L_X$ is the cotangent complex of $X$ over $\Spec k$.
\end{thm}
As a corollary, we can simplify the right-hand side of \eqref{hkr isom mf-intro} and obtain the following. 
\begin{cor} (Corollary \ref{cor:homo}) \label{cor:homo-intro}
    Let $(X,f)$ be a LG model. Then we have the following equivalence of 2-periodic complexes,
    \begin{equation}\label{f2}
        HH_{\bullet}(MF^{\mathrm{coh}}(X,f)) \simeq R\Gamma \left(\Tot\left( \cdots \to \bigoplus_{i~ {\rm even}}\bigwedge^i\L_{I^{DM}X} \xrightarrow{(-)\wedge df\mid_{I^{DM}X}} \bigoplus_{i~ {\rm odd}}\bigwedge^i\L_{I^{DM}X} \cdots \right)\right).
    \end{equation}
\end{cor}
For Hochschild cohomology, we obtain a version of \thmref{thm:homo-intro} which is as follows:
\begin{thm}[\thmref{thm:Hochs cohomo main}]\label{thm:cohomo-intro}
	Let $X$ be a quasi-smooth, finite type, 
 separated, derived $\DM$ stack over a field $k$ of characteristics 0 with  affine stabilizers.
Let $(X,f)$ be an LG model satisfying Assumption \ref{assumption}. 
Then we have the following equivalence of $k((\beta))$-modules:
\begin{equation*}
	HH^{\bullet}(MF^{\mathrm{coh}}(X,f)) \simeq R\Gamma(I^{DM}X, \Sym(\mathbb{L}_{I^{DM}X}[1]) \otimes \pi^*\omega_X^{\vee}[-\dim X])^{B{\mathbb{G}_a}}\otimes_{k[[\beta]]} k((\beta)).
\end{equation*}
where $\pi : I^{DM}X \to X$ is the natural projection and $\omega_X$ is the dualizing complex of $X$.
\end{thm}
As an important step in proving \thmref{thm:homo-intro} and \thmref{thm:cohomo-intro}, we need to prove a Thom-Sebastiani type isomorphism and a duality isomorphism for the category of coherent matrix factorizations, denoted by $MF^{coh}(X,f)$. In \cite{Preygel:2011}, Preygel proved a Thom-Sebastiani type isomorphism for matrix factorization categories corresponding to a LG model $(X,f)$, where $X$ is a smooth $\DM$ stack and $f: X \to \A^1$ is a flat morphism. In this article, we adapt Preygel's proof to the setting of coherent matrix factorization categories. In particular, we prove the following:

\begin{thm}[\thmref{thm:thom-sebastiani}]\label{thm:ts-intro}
Let $(X,f)$ and $(Y,g)$ be two LG models.
Let $i: X_0 \hookrightarrow X$ and $j: Y_0 \hookrightarrow Y$ be the inclusions of the derived zero loci of $f$ and $g$, respectively.
Let $k: (X \times Y)_0 \hookrightarrow X \times Y$ be the inclusion of the derived zero locus of $f \boxplus g$.
Let $\ell: X_0 \times Y_0 \hookrightarrow (X \times Y)_0$ be the natural inclusion.
Let $Z_X$ and $Z_Y$ be closed sub-stacks of $X_0$ and $Y_0$, respectively.
Then we have the following Thom-Sebastiani type equivalence of $k((\beta))$-linear $\infty$-categories:
\begin{align}\nonumber
    MF^{coh, \infty}_{Z_X}(X,f) \otimes_{k((\beta))} MF^{coh, \infty}_{Z_Y}(Y,g) &\xrightarrow{\ell_*(-\boxtimes -)} MF^{coh, \infty}_{Z_X \times Z_Y}(X \times Y, f \boxplus g).
\end{align}  
 On restriction to compact
objects we get the following equivalence of $k((\beta))$-linear $\infty$-categories: 
\begin{align}\nonumber
    MF^{coh}_{Z_X}(X,f) \otimes_{k((\beta))} MF^{coh}_{Z_Y}(Y,g) &\xrightarrow{\ell_*(-\boxtimes -)} MF^{coh}_{Z_X \times Z_Y}(X \times Y, f \boxplus g).
\end{align}  
\end{thm}
As a consequence, we get a duality theorem for the category $MF^{coh}(X,f)$, which is as follows. Preygel \cite[Theorem 4.2.2]{Preygel:2011} proved this when $X$ is a smooth $\DM$ stack with a flat potential $f$.
\begin{thm}[\thmref{thm:duality}]\label{thm:duality-intro}
Let $(X,f)$ be a LG model.
Let $Z$ be a closed sub-stack of $X_0$. By Remark \ref{remk:dualizable}, $MF^{coh, \infty}_{Z}(X,f)$ is a dualizable $\infty$-category.
Then we have the following duality equivalence of $k((\beta))$-linear $\infty$-categories:
\[MF^{coh, \infty}_{Z}(X,f)^{\vee} \cong MF^{coh, \infty}_{Z}(X,-f).\]
\end{thm}

\subsection{Structure of the paper} In  \secref{sec:prelim} we fix notations and conventions that will be used in the text. We recall the definition of the category of coherent matrix factorizations following Preygel \cite{Preygel:2011} and \cite{HLP_hodge} and recall some standard definitions that will be used in the paper. In Section \ref{sec:thom-sebastiani} we prove \thmref{thm:ts-intro}, which is a Thom-Sebastiani type isomorphism for the category of coherent matrix factorizations. In  \secref{sec:duality} we prove \thmref{thm:duality-intro}, which is a duality theorem for the category of relative factorizations. In  \secref{sec:HKR} we prove the main theorem of this text,  \thmref{thm:homo-intro} and  \thmref{thm:cohomo-intro}. 
\subsection{Acknowledgements}
We thank Bertrand T\"oen and Massimo Pippi for helpful discussions and comments on a draft of this paper. We are grateful to Massimo Pippi for pointing out a serious error in the original formulation of the theorem and for suggesting multiple corrections. We also thank ICTS for hosting the program Enumerative Geometry and Categorification (ICTS/EGC2026/07), during which these discussions took place. B.S. gratefully acknowledges the support of ANRF MATRICS grant  ANRF/ARGM/2025/001980/MTR.

\section{Preliminaries}\label{sec:prelim}
In this section, we briefly recall some preliminaries that will be used in the rest of the text. 
\subsection{Notations and conventions}
\begin{notat}\label{notat}
Let $k$ denote an algebraically closed field of characteristic $0$. All stacks in this text are assumed to be separated, finite-type, perfect (see  \defref{defn: perfect stack}) derived Deligne-Mumford ($\DM$) stacks (see Section \ref{sec:homotopical alg}) with affine stabilizers such that the structure sheaf $\cO_X$ is cohomologically bounded.  All functors are assumed to be derived unless mentioned otherwise.  

Let $X$ be a derived $\DM$ stack over $\Spec k$. The underlying classical stack is denoted by $X_{cl}$. Let $\Qcoh (X)$ denotes the stable $\infty$-category of quasi-coherent complexes on $X$ \cite[Chapter 3, Definition.~1.1.4]{Gaitsgory-book-1}. Let $Perf(X)$ denote the full subcategory of $\Qcoh(X)$ consisting of perfect complexes. Let $D^b{Coh}(X)$ denotes the stable $\infty$-category of cohomologically bounded objects $E^{\bullet}$ in $\Qcoh(X)$ such that $H^*(E^{\bullet})$ is a coherent $H^0(\cO_X)$-module. Let $\IndCoh(X)$ denotes the Ind-completion of $D^b{Coh}(X)$ (see \cite[Definition.~5.3.5.1]{HTT}). Let $\sF \in \IndCoh(X)$ and $\sG \in \IndCoh(Y)$, their external tensor product $\sF \boxtimes \sG \in \IndCoh(X\times Y)$ is defined by,
    \begin{equation}\label{eqn:boxtimes}
        \sF \boxtimes \sG := pr_X^*(\sF) \otimes_{\cO_{X\times Y}} pr_Y^*(\sG),
    \end{equation}
    where $pr_X: X\times Y \to X$ and $pr_Y:X \times Y \to Y$ denote the projections to $X$ and $Y$ respectively.
Let $Z \xrightarrow{i} X$ be a closed sub-stack of $X$ then  $\IndCoh_Z(X)$ and  $D^bCoh_Z(X)$ denotes the full subcategories of $\IndCoh(X)$ and $D^bCoh(X)$ consisting of objects supported on $Z$. 
We denote the idempotent completion  of an $\infty$-category $\sC$ by $\ol{\sC}$ (see \cite[Section 5]{HTT}).

Let $DSt_{/k}$ denotes the category of derived $1$-Deligne-Mumford stacks over $\Spec k$ and let $AffDSt_{/k}$ denote the full subcategory of $DSt_{/k}$ consisting of affine derived stacks \cite[Section~2]{champes-affine} and \cite[Section~3.2]{Ben_Zvi_2012}. 
For a topological space $\sT$, we denote its singular co-chain complex with coefficients in $k$ by $C^*(\sT;k)$. An $\infty$-category $\sC$ is said to be $k$-linear if the homotopy category \cite[Section 1.1.2]{HA} of $\sC$ is $k$-linear. For a derived stack $X \in DSt_{/k}$ let $\mathbb{L}_X$ denote the cotangent complex of $X \to \Spec k$ \cite[Section 3.2]{Thesis_Lurie}. For a morphism $f: X\to Y$ between two derived stacks $X$ and $Y$, the induced morphism between the underlying classical stacks is denoted by $f_{cl}: X_{cl} \to Y_{cl}$. For an $\infty$-category $\mathcal{C}$, the category of compact objects is denoted by $\mathcal{C}^c$.

Unless otherwise specified, all fiber products in the text are taken over $\Spec k$. $-\times^h_S -$ denotes the homotopy fiber product \cite[Section 1.2.13]{HTT} over $S$ for any base $S$. We denote the $\infty$-category of spaces by $\sS$ (see \cite[Section 1.2.16]{HTT} ).

Let $\sC$ be a stable $\infty$-category and $\sC_0\xhookrightarrow{i} \sC$ be a full subcategory closed under fiber and cofiber. The quotient $\sC/\sC_0$ is defined by the following homotopy pushout square in the category of stable $\infty$-categories. In this case the quotient is also a stable $\infty$-category \cite[Lemma 1.1.3.3]{HA} (see also \cite[end of Section 2.1]{MR3877165}):
    \begin{equation}\label{eqn:quotient}
        \begin{tikzcd}[cramped]
    	{\sC_0} & {\sC} \\
    	0 & {\sC/\sC_0}.
    	\arrow["i", from=1-1, to=1-2]
    	\arrow[from=1-1, to=2-1]
    	\arrow[from=1-2, to=2-2]
    	\arrow[from=2-1, to=2-2]
        \end{tikzcd}
    \end{equation}
     Let $X$ be a derived stack and $f: X\to \A^1$ be a morphism. The derived zero locus of $f$ is defined as $X_0:= 0 \times^h_{\A^1} X$ given by the  following homotopy pullback square in $DSt_{/k}$,
     \begin{equation}\label{eqn:zero-locus}
        \begin{tikzcd}[cramped]
    	{X_0} & {X} \\
    	0 & {\A^1}.
    	\arrow[ from=1-1, to=1-2]
    	\arrow[from=1-1, to=2-1]
    	\arrow["f",from=1-2, to=2-2]
    	\arrow["0",from=2-1, to=2-2]
        \end{tikzcd}
    \end{equation}
\end{notat}

\subsection{A model for \texorpdfstring{$dgCat^{\infty}_k$}{dgCat-infty-k}} 
We recall a model for $dgCat^{\infty}_k$ following \cite[Section 2.1]{MR3877165}. Let
$dgCat_k$ denote the category of all small $k$-linear dg-categories together with 
$k$-linear dg functors. By \cite[Definition 2.11, 2.12, 2.14]{homotopy-th-of-dg-cats-Tabuada} 
$dgCat_k$ has a cofibrantly generated model category structure where the weak 
equivalences are DK equivalences\cite[Theorem 2.8]{homotopy-th-of-dg-cats-Tabuada}. 
The underlying $\infty$-category of this model category is denoted by $dgCat_k$. 
As every DK equivalence is a Morita equivalence, we can therefore Bousfield localize 
the above model category at Morita equivalences to get a second cofibrantly generated 
combinatorial model category structure. In this model structure, the weak equivalences are Morita 
equivalences. Similarly, the underlying $\infty$-category of this model category is 
denoted by $dgCat^{idem}_k$. 

\subsection{Derived stacks and homotopical algebra } \label{sec:homotopical alg}
To setup notations we  recall some standard definitions in derived algebraic geometry. For a detailed exposition,  see \cite{khan_derived-geom},\cite{Gaitsgory-book-1} and \cite{HAG_2}. 
\subsubsection{Derived stacks}
For a commutative ring $R$, let $dCAlg_R$ denotes the $\infty$-category of derived commutative $R$-algebras \cite[Example 1.2.20]{khan_derived-geom}.
\begin{defn}\cite[Definition. 1.3.6]{khan_derived-geom}
    Let $R$ be a commutative ring. A \emph{derived stack} over $R$ is a functor 
    \[X: dCAlg_R \to \sS\] which satisfies \'etale descent.
\end{defn}
\begin{defn}\cite[Definition. 1.3.10]{khan_derived-geom}
    A derived stack $X$ is called a \emph{derived scheme} if there exists a collection of open immersions $\{U_{\alpha} \hookrightarrow X\}$ where $U_{\alpha}$ are affine derived schemes and the induced morphism $\coprod_{\alpha}U_{\alpha} \twoheadrightarrow X$ is surjective.
\end{defn}
For a derived commutative ring $A$ let $D(A)$ denote the derived $\infty$-category of $A$-modules \cite[Section 1.3.3]{khan_derived-geom}, \cite[Section 1.3.2]{HA}.

\begin{defn}\cite[Definition 2.45 \& Example 2.46]{khan_derived-geom}\label{defn:quasi-smooth}
    A morphism $f:X\to Y$ between derived 1-stacks $X$ and $Y$ is said to be {\em quasi-smooth} if the cotangent complex $\L_f$ is a perfect complex and is of Tor-amplitude $\leq 1$ and $f_{cl}$ is locally of finite presentation. For a quasi-smooth morphism,  $f:X\to Y$, let $\sN_f$ denote the co-normal complex of $f$. It follows from  \cite[pg. 3]{vit_khan} that $\sN_f = \L_f[-1].$
\end{defn}

\subsubsection{Affine stacks}\label{sec:affine stack}
 We briefly recall the definitions of affinization and affine stacks following \cite{champes-affine} and \cite[Section 3]{Ben_Zvi_2012}.
Let $DGA_k$ denotes the $\infty$-category of commutative dg $k$-algebras. Recall the following natural adjunction \cite[Proposition 3.1]{Ben_Zvi_2012}:
\[\cO : DSt_{/k}\leftrightarrows DGA_k^{op} :\Spec.\]
\begin{defn}\cite[Definition. 3.2]{Ben_Zvi_2012}
    The endofunctor 
    \[Aff : DSt_{/k} \to DSt_{/k}\] defined by
    \[Aff(X):=\Spec(\cO_X)\]
    is called the affinization functor. A derived stack $X$ is said to be \emph{affine stack} if the canonical morphism $X\to Aff(X)$ is an equivalence.
\end{defn}
\begin{remk}
    By \cite[Proposition 2.2.7]{champes-affine} it follows that affine stacks are closed under small limits.
\end{remk}
\begin{exms}
 A derived affine scheme is clearly an affine stack.
      The stacks $K(\G_a,m)$ for $m>0$ are affine stacks \cite[Lemma 2.2.5]{champes-affine}. 
\end{exms}
One important result we need in Section \ref{sec:HKR} is the following:
\begin{prop}\cite[Lemma 3.13]{Ben_Zvi_2012}
    The affinization morphism is an equivalence of group-derived stacks
    \begin{equation}\label{affinization of BG-a}
        Aff(S^1) \xrightarrow{\simeq} B\G_a=K(\G_a,1). 
    \end{equation}
\end{prop}

\subsubsection{Homotopy fixed point and homotopy orbit}
In this section, we briefly recall the definitions of homotopy fixed points and homotopy orbits following \cite[Section 2.4]{raksit-hochschildhomologyderivedrham}, \cite{hoyois-homotopy-fixed-points-circle} and \cite[Chapter 1]{thh-scholze}.
Let $\sC$ be a $\infty$-category. Let $G$ be a group. A $G$-equivariant object in $\sC$ is a functor $BG \to \sC$ \cite[Definition. I.1.2]{thh-scholze}. Let the category of all $G$-equivariant objects of $\sC$ be denoted by $\sC^{BG}:= \Fun(BG,\sC).$
\begin{defn}\cite[Definition. I.1.5]{thh-scholze}\label{fixed pt and orbit}
    Let $G$ be a group and $\sC$ be an $\infty$-category. \\
    (i) Assume $\sC$ admits all colimits indexed by $BG$. Then the \emph{homotopy orbits} functor is given by 
    \begin{equation}
        \begin{split}
            (-)_{G} : \sC^{BG} &\to \sC\\
            (F:BG \to \sC) &\mapsto \hocolim_{BG}F.
        \end{split}
    \end{equation}
    (ii) Assume that $\sC$ admits all limits indexed by $BG$. The \emph{homotopy fixed points} functor is given by 
    \begin{equation}
        \begin{split}
            (-)^G : \sC^{BG} &\to \sC\\
            (F:BG \to \sC) &\mapsto \holim_{BG}F.
        \end{split}
    \end{equation}
\end{defn}
In our situation in Section \ref{sec:HKR} we need to work with $S^1$-algebra objects of the $\infty$-category $CAlg_k$ and mixed algebra objects of $CAlg_k$ which is defined as $S^1\mbox{-}CAlg_k$, and $\epsilon\mbox{-}CAlg_k := CoMod_{k\oplus k[-1]}(CAlg_k)$ respectively \cite[Recollection 5.1]{HKR2026}. By \cite[Theorem 4.1]{TV_multi_hkr} we have the following equivalence of $\infty$-categories:
\[A_{\phi} : S^1\mbox{-}CAlg_k \to \epsilon\mbox{-}CAlg_k\]
induced from the isomorphism 
\[\phi:C^*(S^1;k) \simeq k \oplus k[-1].\]
Now, by \cite[Section 5.2]{HKR2026} and \cite[Section 2]{hoyois-homotopy-fixed-points-circle} we have following adjunctions:
\[\begin{tikzcd}[cramped]
	{(Mod_k)^{S^1}} &&& {Mod_k}
	\arrow["{(-)^{S^1}}", shift left=3, from=1-1, to=1-4]
	\arrow["{(-)_{S^1}}"', shift right=5,  from=1-1, to=1-4]
	\arrow["{triv_{S^1}}", from=1-4, to=1-1]
\end{tikzcd}~~~~~~~~\,\,\,\,\,\,\,\,\,\,\,\,\,
\begin{tikzcd}[cramped]
	{\epsilon\mbox{-}Mod_k} &&& {Mod_k}
	\arrow["{Map_{\epsilon\mbox{-}Mod_k}(k,-)}", shift left=3,from=1-1, to=1-4]
	\arrow["{k\otimes_{k[\epsilon]}-}"', shift right=5, from=1-1, to=1-4]
	\arrow["{triv_{\epsilon}}", from=1-4, to=1-1]
\end{tikzcd}\]
We have 
\begin{equation}\label{eqn: fixpt adjunction}
    (-)_{S^1} \dashv triv_{S^1} \dashv (-)^{S^1}, {\rm     and~~~~  } k\otimes_{k[\epsilon]}(-) \dashv triv_{\epsilon} \dashv Map_{\epsilon\mbox{-}Mod_k}(k,-)
\end{equation}
where $\mathrm{triv}_{S^1}$ and $\mathrm{triv}_{\epsilon}$ are functors which equip $M \in Mod_k$ with trivial $S^1$ and $\epsilon$ action respectively, and the equivalence $A_{\phi}$ intertwines left and right adjoints \cite[eqn. 5.2.2]{HKR2026} (also cf. \cite[Notation 2.4.1]{raksit-hochschildhomologyderivedrham}).

\subsection{Matrix factorization category}

We briefly recall some basic definitions from Preygel \cite[Section 3]{Preygel:2011} of Landau-Ginzburg models and matrix factorization categories.

\begin{defn}\cite[Definition. 2.23]{khan-ravi-cohomo-alg-stacks}\label{defn: perfect stack}
    A derived algebraic stack $X$ is said to be a {\em perfect stack} if $\Qcoh(X)$ is compactly generated by the full subcategory $Perf(X)$ of perfect complexes.
\end{defn}

\begin{defn}
A {\em Landau-Ginzburg (LG) model} is a pair $(X, f)$ where $X$ is a derived $\DM$-stack over a field $k$ of characteristic $ 0$ and $f : X \to \A^1$ is a morphism. 
\end{defn}

By abuse of notation we denote $\Spec k$ by $0$, when we consider the inclusion of the closed point $\{0\} \inj \A^1$. Let $\mathbb{B} := 0 \times^h_{\A^1} 0 = \Spec k[\epsilon]/\epsilon^2$ with $\deg \epsilon = +1$.

\begin{equation}\label{eqn:B}
        \begin{tikzcd}[cramped]
    	{\B} & {0} \\
    	0 & {\A^1}.
    	\arrow[ from=1-1, to=1-2]
    	\arrow[from=1-1, to=2-1]
    	\arrow["0",from=1-2, to=2-2]
    	\arrow["0",from=2-1, to=2-2]
        \end{tikzcd}
\end{equation}
Then $\B$ can be seen as a commutative monoid object in $DSt_{/k}$ by the following two operations:
\begin{enumerate}
    \item Define $\mu : \B \times \B \to \B$ by \[\B \times \B = (0 \times_{\A^1} 0) \times (0 \times_{\A^1} 0) \simeq 0 \times_{\A^1} 0 \times_{\A^1} 0 \xrightarrow{p_{13}} \B\]
    where $p_{13}$ is the projection to the first and the third factor.
    \item Define $+:\B \times \B \to \B$ induced from $+:\A^1 \times \A^1 \to \A^1$, see \cite[Construction 3.1.1]{Preygel:2011} for more details.
\end{enumerate}
 Then $\cO_{\B}$ is identified with $C^*(S^1;k) \simeq k \oplus k[-1]$ as the $\E_{\infty}$-coalgebra in the category of dg-algebras. Then taking the cobar construction one gets $k[[\beta]] \simeq C^*(BS^1;k)$ as $\E_{\infty}$-algebra where $\beta$ is a variable of degree 2. As a dg algebra $k[[\beta]]$ is 
\[\cdots \to 0 \to 0 \to k \to 0 \to k \to 0 \to k \to 0 \to \cdots,\]
where the first $k$ is in degree 0. 

Consider the category $\IndCoh(\B)$ equipped with a symmetric monoidal structure given by the convolution operation, which is defined by $\sF \circ \sG := \mu_*(\sF \boxtimes \sG)$.
Then there exists a symmetric monoidal equivalence \cite[Proposition 3.1.4]{Preygel:2011}
\[(\IndCoh(\B), \circ) \simeq (k[[\beta]]\mbox{-}\Mod, \otimes_{k[[\beta]]}).\]

Let $(X,f)$ be a LG model. By \cite[Construction 3.5]{Preygel:2011}, there exists an action of $\B$ on the derived zero locus $X_0$ of $f$. The action can be seen explicitly by the following Cartesian cube \cite[Section 2.3.35]{MR3877165}. Here, $pr_{\B}$ and $pr_{X_0}$ are projections to $\B$ and $X_0$ respectively.

\[\begin{tikzcd}[cramped]
	{X_0\times \B} && {X_0} & \\
	& {X_0} && X \\
	\B && 0 \\
	& 0 && {\A^1 }
	\arrow["v"{description}, from=1-1, to=1-3]
	\arrow["{pr_{X_0}}"{description}, from=1-1, to=2-2]
	\arrow["{pr_{\B}}"{description}, from=1-1, to=3-1]
	\arrow["i"{description}, from=1-3, to=2-4]
	\arrow[from=1-3, to=3-3]
	\arrow["i"{description, pos=0.3}, from=2-2, to=2-4]
	\arrow[from=2-2, to=4-2]
	\arrow["f"{description}, from=2-4, to=4-4]
	\arrow[from=3-1, to=3-3]
	\arrow[from=3-1, to=4-2]
	\arrow[from=3-3, to=4-4]
	\arrow["0"{description}, from=4-2, to=4-4]
\end{tikzcd}\]
For $\sF \in \IndCoh(\B)$ the action of $\sF$ on $\IndCoh(X_0)$ is given by,
\begin{equation}\label{eqn:action}
\sF \cdot \sG := v_*(pr_{\B}^*(\sF) \otimes pr_{X_0}^*(\sG)).
\end{equation}
Therefore, $\IndCoh(X_0)$ has a $\IndCoh(\B)$-module structure. 
Since  $pr_{\B}$ is flat, it has finite Tor dimension. Therefore, $\pr_{\B}^*$ preserves the property of bounded cohomological dimension. Hence, the action restricts to compact objects and
we get an action of $D^bCoh(\B)$ on the category $D^bCoh(X_0)$. By the proof of \cite[Proposition 3.1.4, second last paragraph at page 13]{Preygel:2011} we have, $RHom_{\cO_{\B}\mbox{-}\Mod}(k,-)$ gives an equivalence between $D^bCoh(\B)$ and $Perf(k[[\beta]])$. This gives a  $k[[\beta]]$-linear structure on the stable $\infty$-category $D^bCoh(\B)$.

\begin{defn}\cite[Section 2.4]{Preygel:2011}
 Let $(X, f)$ be a LG model and $X_0$ be the derived zero locus of $f$. 
 The {\em pre-matrix factorization category} denoted by $PreMF(X,f)$ is defined as the $k[[\beta]]$-linear $\infty$-category where the underlying $k$-linear category is $D^bCoh(X_0)$
and $\beta$ acts via the natural transformation induced by the
$S^1$-action on $D^bCoh(X_0)$ constructed in \cite[Construction 3.1.1]{Preygel:2011}.
\end{defn}

\begin{defn}\cite[Section 2.4]{Preygel:2011}\label{defn:mf absolute}
Let $(X, f)$ be a LG model. 
We define the {\em matrix factorization category} denoted by 
$MF(X, f)$ as
\[MF(X,f) := PreMF(X,f) \otimes_{k[[\beta]]} k((\beta)).\]
\end{defn}

\begin{defn}\label{defn:sing cat}
    For an algebraic stack $X$ we define the following singularity categories:
    \[Sing(X) := D^bCoh(X)/Perf(X)\] and
    \[Sing^{\infty}(X) := \IndCoh(X)/QCoh(X)\]
    as $k$-linear $\infty$-categories where the quotient is in the sense of \eqref{eqn:quotient}.
\end{defn}

\begin{remk}
    To make sense of \defref{defn:sing cat} we need to assume $\cO_X$ is cohomologically bounded; otherwise, $Perf(X)$ need not be a subcategory of $D^bCoh(X)$ (see \cite[Remark 3.2.7]{motivic-pippi}).
\end{remk}

For a LG model $(X,f)$, the relationship between the category of matrix factorizations and the singularity category of the derived zero locus $X_0$ is given by Orlov's theorem \cite[Proposition 3]{Orlov-mf}. In particular, for smooth $\DM$ stacks, one has the following result.

\begin{prop}\cite[Proposition 3.4.1]{Preygel:2011}\label{prop. orlov}
Let $(X,f)$ be a LG model with $X$ a smooth $\DM$ stack and $f$ flat. Then the natural $k$ -linear functor 
\[PreMF(X,f) \to MF(X,f)\]
factors through the quotient functor $D^bCoh(X_0) \to Sing(X_0)$ 
and the induced functor 
\[Sing(X_0) \to MF(X,f)\]
is an idempotent completion.
\end{prop}

\subsection{Coherent matrix factorization categories}
In this subsection, we define the coherent matrix factorization category associated with a LG model. The matrix factorization category introduced in Definition \ref{defn:mf absolute} behaves well when the stack $X$ is smooth; however, for singular stacks, Proposition \ref{prop. orlov} need not hold. 
Coherent matrix factorization categories are the correct setting to generalize Proposition \ref{prop. orlov} for singular schemes. See, for example \cite{pippi}, \cite{coherent-analogues}. 
Motivated by these observations, we consider the coherent matrix factorization category to prove a HKR-type theorem for LG models in the non-smooth setting.

\begin{defn}[Coherent Matrix Factorization Categories]\label{defn:coh factorizations}

    \begin{enumerate}
    \item[]
        \item Define the  {\em coherent big pre matrix factorization category} $PreMF^{coh, \infty}(X,f)$ corresponding to a LG model $(X,f)$ as the $k[[\beta]]$-linear $\infty$-category with the underlying $k$-linear category $\IndCoh(X_0)$.  
        \item Define the {\em coherent big matrix factorization category} as $$MF^{coh, \infty}(X,f) := PreMF^{coh, \infty}(X,f) \otimes_{k[[\beta]]} k((\beta)).$$
        \item Define the {\em coherent big pre matrix factorization category supported at $Z$}, denoted by 
        $PreMF^{coh, \infty}_Z(X,f)$ for a closed sub-stack $Z \subset X_0$ as the $k[[\beta]]$-linear
        $\infty$-category with the underlying $k$-linear category $\IndCoh_Z(X_0)$, and the {\em coherent big matrix factorization category supported at $Z$} as $$MF^{coh, \infty}_Z(X,f) := PreMF^{coh, \infty}_Z(X,f) \otimes_{k[[\beta]]} k((\beta)).$$
        \item Define the coherent {\em pre matrix factorization category supported at $Z$} denoted by 
        $PreMF^{coh}_Z(X,f)$ as the $k[[\beta]]$-linear $\infty$-category with 
        the underlying $k$-linear category 
        is $D^bCoh_Z(X_0)$.
        \item Define the {\em coherent matrix factorization category supported at $Z$} denoted by 
        $$MF^{coh}_Z(X,f) :=PreMF^{coh}_Z(X,f)\otimes_{k[[\beta]]} k((\beta)).$$
    \end{enumerate}
    Henceforth, for $Z=X_0$ we write $MF^{coh}(X,f)$ and $MF^{coh, \infty}(X,f)$ instead of $MF^{coh}_Z(X,f)$ and $MF^{coh, \infty}_Z(X,f)$ respectively.
\end{defn}

\begin{lem}\label{lem:beta torsion}
    The category of $\beta$-torsion elements for the action of $k[[\beta]]$ on $D^bCoh(X_0)$ is given by the category generated by $i^*D^bCoh(X)$ up to idempotent completion.
\end{lem}

\begin{proof}
    Let $\mathcal{F} \in D^bCoh(X_0)$ be a $\beta$-torsion element. 
        Therefore, the action of $\beta^m$ is null homotopic on $\sF$ for some $m\in \mathbb{N}$. We will apply induction on $m$. 
    
    By \cite[Example 3.1.13]{Preygel:2011} we have a triangle
    \[i^*i_*\sF \to \sF\xrightarrow{\beta}  \sF[2] \xrightarrow{+1}.\]
    Therefore, $\hofib(\beta) \in i^*D^bCoh(X)$. Now, consider the following three triangles:
    \[\hofib(\beta^{n-1}) \to \sF \xrightarrow{\beta^{n-1}} \sF[2n-2] \xrightarrow{+1},\]
    \[\hofib(\beta)[2n-2] \to \sF[2n-2] \xrightarrow{\beta} \sF[2n] \xrightarrow{+1},\] and
    \[\hofib(\beta^{n}) \to \sF \xrightarrow{\beta^{n}} \sF[2n] \xrightarrow{+1}.\]
    By octahedral axiom \cite[Definition 1.1.4.5, TR4]{HA} we get the following triangle:
    \begin{equation}\label{triangle 1}
        \hofib(\beta^{n-1}) \to \hofib(\beta^n) \to \hofib(\beta)[2n-2] \xrightarrow{+1}.
    \end{equation}
    Now, assume $\hofib(\beta^{n-1}) \in i^*D^bCoh(X)$. Then form \eqref{triangle 1} we have,
    \[\hofib(\beta^n) \in \langle i^*D^bCoh(X) \rangle.\]
    Therefore, by induction $\hofib(\beta^m) \in \langle i^*D^bCoh(X) \rangle.$ By our assumption, $\beta^m$ is homotopic to 0. Thus, 
    \[\hofib(\beta^m) \simeq \hofib(\sF \xrightarrow{0}\sF[2m]) \simeq \sF \oplus \sF[2m-1].\]
    Therefore, we get
    \[\sF \in \overline{\langle i^*D^bCoh(X) \rangle}.\]

    Conversely, let $\mathcal{E} \in D^bCoh(X)$. Then $i^*\mathcal{E}\in i^*D^bCoh(X) \subseteq D^bCoh(X_0).$ By \cite[Example 3.1.13]{Preygel:2011} we have the following triangle:
    \[i^*i_*i^*\mathcal{E} \xrightarrow{c} i^*\mathcal{E} \xrightarrow{\beta} i^*\mathcal{E}[2] \xrightarrow{+1},\]
    where $c$ is induced by the counit of the adjunction $(i^*,i_*)$. Let $u$ be the unit of the adjunction $(i^*,i_*)$. Then $i^*u$ gives a section of $c$. Therefore, $\beta = \beta \circ (c\circ i^*u) = (\beta \circ c)\circ i^*u = 0$. Thus $i^*\mathcal{E}$ is a $\beta$-torsion element. Hence we have,
    \[\overline{\langle i^*D^bCoh(X) \rangle} \subseteq \overline{\langle \beta\mbox{-}torsion\rangle}.\]
    This completes the proof.
\end{proof}

The following lemma is analogous to \cite[Theorem 2.7]{coherent-analogues}, where Efimov and Positselski proved that $D^b_{Sing}(X_0/X)$ is equivalent to the absolute derived category of coherent factorizations, and in \cite[Theorem 2.8]{coherent-analogues} authors proved the large version of relative singularity category \cite[Section 2.8]{coherent-analogues} is equivalent to co-derived category \cite[Section 1.3]{coherent-analogues} of quasi-coherent factorizations. For smooth $\DM$ stack $X$ Preygel proved analogous result in \cite[Proposition 3.4.1]{Preygel:2011}.

\begin{lem}
    The natural functor 
    \[PreMF^{coh}(X,f) \to MF^{coh}(X,f) \]
    factors as
\[\begin{tikzcd}[cramped]
	{PreMF^{coh}(X,f)} &&& {MF^{coh}(X,f)} \\
	& {D^b_{Sing}(X_0/X)} & {\dfrac{D^bCoh(X_0)}{\langle i^*D^bCoh(X)\rangle}}
	\arrow[from=1-1, to=1-4]
	\arrow[from=1-1, to=2-2]
	\arrow["{:=}"{description}, draw=none, from=2-2, to=2-3]
	\arrow[from=2-3, to=1-4]
\end{tikzcd}\]
and the induced map from $D^b_{Sing}(X_0/X) \to MF^{coh}(X,f)$ is an idempotent completion.
\end{lem}

\begin{proof}
    By \lemref{lem:beta torsion} we get the required factorization. 
    Consider the triangle, 
    \[Perf(\mathbb{B}) \to D^bCoh(\B) \to Sing(\B).\]
    By \cite[Lemma 3.1.9]{Preygel:2011}, we get that this can be identified with the following triangle:
    \[\left\{\beta\mbox{-}{\rm torsion~elements~in~} k[[\beta]]\mbox{-}Mod\right\} \to Pref~k[[\beta]] \to Perf~k((\beta)).\]
    Applying $D^bCoh(X_0) \otimes_{D^bCoh(\B)}(-)$ we get the following triangle:
    \begin{equation}\label{equation 1}
        D^bCoh(X_0)\otimes_{D^bCoh(\B)}Perf(\B) \to D^bCoh(X_0) \to D^bCoh(X_0)\otimes_{D^bCoh(\B)}Sing(\B).
    \end{equation}
    Now, again by \lemref{lem:beta torsion} we get,
    \[D^bCoh(X_0)\otimes_{D^bCoh(\B)}Perf(\B) \simeq \langle i^*D^bCoh(X)\rangle.\]
    Hence by triangle \eqref{equation 1} we get,
    \[D^bCoh(X_0)\otimes_{D^bCoh(\B)}Sing(\B) \simeq \dfrac{D^bCoh(X_0)}{\langle i^*D^bCoh(X)\rangle}\]
    up to idempotent completion. Hence, 
    \[MF^{coh}(X,f) \simeq \overline{\left(\dfrac{D^bCoh(X_0)}{\langle i^*D^bCoh(X)\rangle}\right)}.\]
\end{proof}

\subsection{Hochschild homology and cohomology}
In this subsection  we recall the definition and some properties of  the Hochschild homology, following
 \cite[Section 2.2]{chen} and \cite[Section 5.5]{HA}. 
\begin{defn}\cite[Definition 2.2.1]{chen}
    Let $\sD$ be a symmetric monoidal $\infty$-category with monoidal unit 
	$1_{\otimes}$. Let $X\in \sD$ be a dualizable object with dual $X^{\vee}$. Let 
	 $\eta : 1_{\otimes} \to X\otimes X^{\vee}$ denote co-evaluation map, and  $\epsilon:X\otimes X^{\vee} \to 1_{\otimes}$ denote the evaluation 
	map. Define the \emph{dimension} of 
	$X$ as
    \[\dim (X):= \epsilon \circ \eta \in \End_{\sD}({1_{\otimes}}).\]
\end{defn}
By \cite[Proposition 4.6.1.10]{HA}, the space of dualizing structures on $X$ is contractible, 
so $\dim(X)$ is defined uniquely up to a unique isomorphism in the homotopy category 
$Ho(\End_{\sD}(1_{\otimes}))$.
\begin{notat}\label{notat:1}
    Let $Pr^L$ denote the $\infty$-category of presentable stable $\infty$-categories.
	 By \cite[Proposition 4.8.1.15]{HA} the $\infty$-category $Pr^L$ admits a symmetric 
	 monoidal structure. By \cite[Proposition 5.5.3.8]{HTT}, $\Pr^L$ has a internal mapping 
	 space denoted by $Fun^L(-,-)$. Let $Pr^L_k$ denote the $\infty$-category of 
	 $k$-linear presentable stable $\infty$-categories. Furthermore, this tensor 
	 product induces a symmetric monoidal structure on $Pr^L_k$. Let $Pr^{L,w}_{k,\vee}$ 
	 be the $\infty$-category of $k$-linear, dualizable, presentable, stable $\infty$ 
	 categories with left adjoint functors which preserve compact objects. For a $\sD \in Pr^L$, $\sD^w$ denotes the category of compact objects of $\sD$. Let $Vect_k$ denotes the derived $\infty$-category of $k$-vector spaces.
\end{notat}

\begin{defn}\cite[Definition 2.2.6]{chen}
    The \emph{Hochschild homology} functor is defined as
    \[HH_{\bullet} := \dim : Pr^{L,w}_{k,\vee} \to Fun^L_k(Vect_k, Vect_k) \simeq Vect_k.\]
\end{defn}

\begin{defn}\cite[Definition 2.2.10]{chen}\label{defn:HH}
    Let $\sC_{k}^{st}$ denote the $\infty$-category of small stable $k$-linear $\infty$-categories. 
    Define Hochschild homology to be 
    \[HH_{\bullet}:= \dim \circ \Ind : \sC_{k}^{st} \to Fun^L_k(Vect_k,Vect_k)\simeq Vect_k,\]
    more precisely for $\sD \in \sC_{k}^{st}$ $HH_{\bullet}(\sD)$ be the image of $k$ under the composition,
    \[Vect_k \xrightarrow{coev} Fun^L_k(\sD,\sD) \xrightarrow{\simeq} \sD^{\vee} \otimes \sD \xrightarrow{ev} Vect_k.\]
\end{defn}

\begin{remk}\cite[Remark 2.2.11]{chen}
    If $\sD$ is a compactly generated category, then by definition 
	$HH_{\bullet}(\sD) = HH_{\bullet}(\sD^{c})$. Hence, $HH_{\bullet}(MF^{coh, \infty}(X,f)) = HH_{\bullet}(MF^{coh}(X,f))$.
\end{remk}

For a $k$-dg category $C$ let $C\mbox{-}{\rm mod}$ be the category of $C$ dg-modules. 
Then $Fun^L_k(C\mbox{-}{\rm mod},C\mbox{-}{\rm mod}) \simeq C\otimes C^{op}$-mod by dg-Morita theory \cite{derived-morita-theory}. Under this equivalence, the 
co-evaluation $k$-mod $\to C\otimes C^{op}$-mod corresponds to the functor 
$C\otimes_k -$ and the evaluation corresponds to $-\otimes_{C\times C^{op}}-$ 
\cite[Example 2.2.17]{chen}. Therefore, the Hochschild homology 
\[HH_{\bullet}(C\mbox{-}\Mod) = C\otimes_{C\otimes C^{op}}^L C,\] 
coincides with  the usual definition of Hochschild homology 
\cite[Section 4.6]{pippi} given by
\[HH_{\bullet}(C) = H^{-\bullet}\left(id_C \bigotimes_{C\otimes C^{op}} id_C\right)\]
where $id_C$ denotes the identity $C\otimes C^{op}$ bi-module.

\begin{defn}\cite[Definition 5.3 \& 5.4]{iwanari-hochschild-cohomology}\label{defn:cohomo}
    Let $\sD \in \sC_{k}^{st}$. Then \emph{Hochschild cohomology} of $\sD$ is denoted by $HH^{\bullet}(\sD)$ and defined as
    \[HH^{\bullet}(\sD) := Hom_{\Fun^L(\Ind\sD,\Ind\sD)}(id,id) \in Vect_k.\]
\end{defn}

\subsection{Loop stack and orbifold inertia stack}
In this section we recall from \cite[Section 1 \& 2]{HKR2026} some standard constructions and results on a derived $\DM$ stack 
$X$ namely the loop stack $\mathcal{L} X$ and the derived orbifold inertia stack $I^{DM}X$. 
Let 
$\Map(-,-)$ denote the derived mapping stack functor in the category of derived stacks, defined as 
\[\Map(X,Y)(T) := Map_{DSt_{/k}}(X\times T, Y).\]

\begin{defn}\cite[Equation 3.0.2]{HKR2026}
    Let $X$ be a derived stack. The \emph{loop stack} of $X$ denoted by $\mathcal{L}X$ is defined by the following homotopy pullback square in $DSt_{/k}$
    \begin{equation}\label{eqn:loop stack}
        \begin{tikzcd}[cramped]
    	{\mathcal{L}X} & {X} \\
    	X & {X \times X}.
    	\arrow[ from=1-1, to=1-2]
    	\arrow[from=1-1, to=2-1]
    	\arrow["\Delta",from=1-2, to=2-2]
    	\arrow["\Delta"',from=2-1, to=2-2]
        \end{tikzcd}
\end{equation}
This is also equivalent to the derived mapping 
	stack $\Map(S^1, X)$ where $S^1$ is the constant derived stack associated to the topological
	circle. 
\end{defn}

\begin{defn}\cite[Section 4.2]{HKR2026}\label{defn:tangent stack}
	Let $X$ be a derived $\DM$ stack. 
	The shifted tangent bundle $T[-1]X$ is defined as,
	\[T[-1]X := \Spec_X(\Sym(\mathbb{L}_X[1])).\]
\end{defn}

In characteristics 0, $T[-1]X = \Map(B\G_a,X)$ \cite[Theorem 4.8]{HKR2026}.

\begin{defn}\cite[Definition. 3.1]{HKR2026} \cite{khan-ravi-prep}\label{defn:orbifold inertia}
	Let $X$ be a derived $\DM$ stack. Let $C_i$ be the cyclic group of order $i$ for every integer $i>0$. If $i$ divides $j$, then there is a canonical group homomorphism $C_j \twoheadrightarrow C_i$. 
    This group homomorphism induces a morphism $\Map(BC_j,X) \to \Map(BC_i,X)$ in $DSt_{/k}$. 
    The orbifold inertia stack of $X$ is defined as
	\[I^{DM}X := \colim_i \Map(BC_i,X)\]
	where the colimit is in the category $DSt_{/k}$.
\end{defn}

\begin{thm}(HKR for derived $\DM$ stacks, \cite[Proposition 5.6]{HKR2026})\label{HKR-fu+}
    For a derived $\DM$ stack $X$ locally almost of finite presentation over a 
	commutative $\mathbb{Q}$-algebra, there is a canonical equivalence of derived $\DM$ 
	stacks over $X$ which can be lifted to an equivalence over $S^1 \times X$:
    \begin{equation*}
        \xymatrix{
        T[-1] I^{DM} X \ar[rr]^{\widehat{aff}^*}_{\simeq} \ar[dr]&  &\mathcal{L} X \ar[dl] \\
        &X&
        }
    \end{equation*}
    This equivalence is furthermore functorial in $X$.
\end{thm}

\section{Thom-Sebastiani for coherent matrix factorizations}\label{sec:thom-sebastiani}

In this section, our goal is to prove a Thom-Sebastiani type theorem (Theorem \ref{thm:thom-sebastiani}) for the coherent matrix factorization category $MF_Z^{coh}(X,f)$ corresponding to a LG model $(X,f)$. For smooth $\DM$ stacks, the result is due to Preygel \cite[Theorem 4.1.3]{Preygel:2011}.
This is going to be an important step towards the proof of the HKR-type theorem in Section \ref{sec:HKR}. 
Let $(X,f)$ and $(Y,g)$ be two LG models. Let $$(f, g): X\times Y \to \A^1 \times \A^1$$ be the morphism induced by $f\circ pr_X, g \circ pr_Y : X\times Y \to \A^1 \times \A^1.$ Then define
\begin{equation}\label{eqn:f box g}
    f\boxplus g : X\times Y \xrightarrow{(f\circ pr_X, g \circ pr_Y)}\A^1 \times \A^1 \xrightarrow{+} \A^1.
\end{equation}
This gives a symmetric monoidal structure on the category of LG models over $\Spec k$, denoted by $LG_k$ \cite[Construction 2.4]{MR3877165}.
Now we prove the main theorem of this section.

\begin{thm}\label{thm:thom-sebastiani}
Let $(X,f)$ and $(Y,g)$ be two LG models.
Let $i: X_0 \hookrightarrow X$ and $j: Y_0 \hookrightarrow Y$ be the inclusions of the derived zero loci of $f$ and $g$, respectively.
Let $k: (X \times Y)_0 \hookrightarrow X \times Y$ be the inclusion of the derived zero locus of $f \boxplus g$.
Let $\ell: X_0 \times Y_0 \hookrightarrow (X \times Y)_0$ be the natural inclusion.
Let $Z_X$ and $Z_Y$ be closed sub-stacks of $X_0$ and $Y_0$, respectively.
Then we have the following Thom-Sebastiani type equivalence of $k((\beta))$-linear $\infty$-categories:
\begin{align}\label{ts-cpt}
    MF^{coh}_{Z_X}(X,f) \otimes_{k((\beta))} MF^{coh}_{Z_Y}(Y,g) &\xrightarrow{\ell_*(-\boxtimes -)} MF^{coh}_{Z_X \times Z_Y}(X \times Y, f \boxplus g).
\end{align}  
Taking Ind completion, we get,
\begin{align}
    MF^{coh, \infty}_{Z_X}(X,f) \otimes_{k((\beta))} MF^{coh, \infty}_{Z_Y}(Y,g) &\xrightarrow{\ell_*(-\boxtimes -)} MF^{coh, \infty}_{Z_X \times Z_Y}(X \times Y, f \boxplus g).
\end{align}  
\end{thm}

\begin{proof}
Let $R=k[[\beta]]$, $K=k((\beta))$, and set
\[
    h_+=f\boxplus g,
    \qquad
    W=(X\times Y)_0=h_+^{-1}(0),
    \qquad
    h_-=(f\boxminus g)\mid_W.
\]
We first prove the corresponding $R$-linear equivalence
\begin{equation}\label{eqn:ts-premf}
    PreMF^{coh}_{Z_X}(X,f)\otimes_R PreMF^{coh}_{Z_Y}(Y,g)
    \xrightarrow{\ \simeq\ }
    PreMF^{coh}_{Z_X\times Z_Y}(X\times Y,h_+).
\end{equation}

Let $\beta_X$ and $\beta_Y$ denote the two cohomology operators on
$D^bCoh_{Z_X\times Z_Y}(X_0\times Y_0)$ induced by $f$ and $g$,
respectively.  Similarly, let $\beta_+$ and $\beta_-$ denote the
operators induced by $h_+$ and $h_-$.  The linear change of coordinates
\[
    \A^1\times\A^1\longrightarrow\A^1\times\A^1,
    \qquad (u,v)\longmapsto (u+v,u-v),
\]
is an automorphism because $\operatorname{char}(k)=0$.  It identifies the
simultaneous derived zero locus of $(f,g)$ with that of $(h_+,h_-)$ and
identifies the two pairs of operators by
\begin{equation}\label{eqn:ts-beta-change}
    \beta_+\longmapsto\beta_X+\beta_Y,
    \qquad
    \beta_-\longmapsto\beta_X-\beta_Y.
\end{equation}
This is the same $\B^2$-equivariant change of variables as in
\cite[Lemma~4.1.2]{Preygel:2011}.

By 
\cite[Proposition~B.3.2]{Preygel:2011}, exterior product gives an
equivalence
\begin{equation}\label{exterior tensor}
    D^bCoh_{Z_X}(X_0)\otimes_k D^bCoh_{Z_Y}(Y_0)
    \xrightarrow{\ \simeq\ }
    D^bCoh_{Z_X\times Z_Y}(X_0\times Y_0).
\end{equation}
The two commuting $\B$-actions make this equivalence
$k[[\beta_X,\beta_Y]]$-linear.  Taking the relative tensor product over
$R$ amounts to identifying $\beta_X$ and $\beta_Y$.  After the change of
variables \eqref{eqn:ts-beta-change}, this sets $\beta_-=0$ and identifies
$\beta_+$ with $2\beta$.  Rescaling $\beta$ by the unit $2\in k$ we get
\begin{align}\label{eqn:ts-change-of-actions}
    &PreMF^{coh}_{Z_X}(X,f)\otimes_R PreMF^{coh}_{Z_Y}(Y,g)\simeq
    PreMF^{coh}_{Z_X\times Z_Y}(W,h_-)
       \otimes_{k[[\beta_-]]} k,
\end{align}
where the right-hand side retains its $k[[\beta_+]]$-linear structure.
Here the derived zero locus of $h_-$ in $W$ is identified with
$X_0\times Y_0$, and its inclusion in $W$ is $\ell$.

We now apply the supported form of
\cite[Corollary~3.2.4]{Preygel:2011} to the LG model $(W,h_-)$.  Its bar
construction gives an equivalence
\begin{equation}\label{eqn:ts-specialize-beta-minus}
    \ell_*:
    PreMF^{coh}_{Z_X\times Z_Y}(W,h_-)
       \otimes_{k[[\beta_-]]} k
    \xrightarrow{\ \simeq\ }
    D^bCoh_{Z_X\times Z_Y}(W).
\end{equation}
Fully faithfulness follows
from the triangle relating $\ell^*\ell_*$ to the operator $\beta_-$, and
essential surjectivity follows from the finite filtration by powers of the
ideal of the derived zero locus.  The remaining $\beta_+$-action on the
target of \eqref{eqn:ts-specialize-beta-minus} is precisely the action
defining
$PreMF^{coh}_{Z_X\times Z_Y}(X\times Y,h_+)$.

From \eqref{exterior tensor} and \eqref{eqn:ts-specialize-beta-minus}, the functor in \eqref{eqn:ts-premf} is  $\ell_*(-\boxtimes-)$.  As in the proof of
\cite[Theorem~4.1.3]{Preygel:2011}, consider the augmented simplicial
derived stack
\[
    \mathcal X_{-1}=W,
    \qquad
    \mathcal X_n=X_0\times\B^{n-1}\times Y_0\quad(n\geq 0).
\]
The face maps add the corresponding null-homotopies, the degeneracy maps
insert the identity loop, and the augmentation adds the null-homotopies of
$f$ and $g$.  Applying coherent complexes with support and pushforward to
$\mathcal X_\bullet\to W$ gives the simplicial bar construction computing
the relative tensor product over $D^bCoh(\B)$.  Its augmentation is
$\ell_*(-\boxtimes-)$.  Since the augmented diagram is $\B$-equivariant,
this functor is $R$-linear.  Combining
\eqref{eqn:ts-change-of-actions} and
\eqref{eqn:ts-specialize-beta-minus} proves
\eqref{eqn:ts-premf}.

Finally, extend scalars from $R$ to $K$.  Associativity and base change for
relative tensor products give
\begin{align}
    &\bigl(PreMF^{coh}_{Z_X}(X,f)\otimes_R
       PreMF^{coh}_{Z_Y}(Y,g)\bigr)\otimes_R K \simeq
    MF^{coh}_{Z_X}(X,f)\otimes_K MF^{coh}_{Z_Y}(Y,g),
\end{align}
while the right-hand side of \eqref{eqn:ts-premf}, after the same scalar
extension, is
$MF^{coh}_{Z_X\times Z_Y}(X\times Y,f\boxplus g)$.  This proves
\eqref{ts-cpt}.  Taking Ind-completions, and observing that taking Ind-completion
commutes with these tensor products, proves the second equivalence.    This completes the proof.
\end{proof}

\section{Duality for coherent matrix factorizations}\label{sec:duality}
In this section, we prove duality theorems for coherent matrix factorization categories.  Throughout this section, fix the following.
\begin{notat}
    Let $(X,f)$ be a LG model where $X$ is a derived $\DM$ stack as in Notation \ref{notat} and let $Z$ be a closed sub-stack of $X$. For a ring object $\sR$ in an $\infty$-category $\sC$ the category of  $\sR$-module objects in $\sC$ is denoted by $Mod_{\sR}(\sC).$ Let
    $RHom_C(-,-)$ denotes the 0-th homotopy group of the mapping space between two objects in the $\infty$-category $C$, and $R\sHom_{\IndCoh(X_0)}(-,-)$ denotes the inner hom in $\IndCoh(X_0)$ \cite[Notation 2.5.5]{Preygel:2011}.
\end{notat}
\begin{remk}\label{remk:dualizable}
    By \defref{defn:coh factorizations}, $MF^{coh, \infty}_{Z}(X,f)$ is a compactly generated category. 
    Since a compactly generated category is dualizable \cite[Theorem D.7.0.7]{sag}, $MF^{coh, \infty}_{Z}(X,f)$ is a dualizable $\infty$-category. 
\end{remk}

\begin{defn}\label{rmk:pp}
    Let $\sC$ and $\sD$ be two symmetric monoidal $\IndCoh(\B)$-module categories. Let $F:\sC \times \sD \to \IndCoh(\B)$ be a bifunctor \cite[Definition. 2.2.5.3]{HA}.
 $F$ is said to be a perfect pairing if the induced functor
\[\sC \to \Fun^L_{\IndCoh(\B)}(\sD, \IndCoh(\B))\]
is an equivalence \cite[pg. 27]{Preygel:2011}, \cite[Lemma 7.2.4.6]{HA}.
\end{defn}

Now, we recall some  facts about Grothendieck duality from \cite[Section 4.4]{gaitsgory} and \cite[Section 3.1.1, \& Appendix B]{hlp-theta-strata}.
Let $X\xrightarrow{p} \Spec k$ be an almost finitely presented, quasi-smooth derived $\DM$ stack with affine stabilizers. Then there exists a dualizing complex $\omega_X := p^!(k) \in \IndCoh(X),$ \cite[Section 5]{indcoherentsheaves} where $p^!:\IndCoh(\Spec k) \to \IndCoh(X)$ be the upper shriek pullback functor . Then  $RHom_{\IndCoh(X)}(-,-)$ defines a functor
\begin{equation}\label{eqn:gd}
    \D(-):=RHom_{\IndCoh(X)}(-,\omega_X): \IndCoh(X) \to \Qcoh(X)^{op}
\end{equation}
which restricts to an equivalence
\begin{equation}\label{eqn:gdcoh}
    \D(-):D^bCoh(X) \to D^bCoh(X)^{op},
\end{equation}
\cite[Section 3.1.1]{hlp-theta-strata}, \cite[Section 4.4]{gaitsgory}.
Since $X$ is quasi-smooth $\omega_X \in Perf(X)$ \cite[Cor. 2.2.7]{singular-support}.
Let $(X,f)$ be an LG model and $X_0 \xrightarrow{i} X$ be the derived zero locus of $f$. Since $i$ is quasi-smooth $i^!(-) = i^*(-) \otimes \omega_i$, where $\omega_i$ is relative dualizing complex $\omega_i = \det(\L_i)$ \cite[Lemma B.0.2]{hlp-theta-strata}. Therefore, the dualizing complex of $X_0$, $\omega_{X_0}= i^!\omega_X$ is a perfect complex.

\subsection{Realization of structure sheaf of \texorpdfstring{$X$}{X} as a factorization of 0}\label{sec:o_x as facto}
Consider the following derived pullback square. 
\begin{equation}\label{diag:d5}
    \begin{tikzcd}[cramped]
	X && \\
	& {X_{00}} & X \\
	& 0 & {\A^1}
	\arrow["{i_0'}", dashed, from=1-1, to=2-2]
	\arrow["id", bend left, from=1-1, to=2-3]
	\arrow["0"', bend right, from=1-1, to=3-2]
	\arrow["{i_0}", from=2-2, to=2-3]
	\arrow[from=2-2, to=3-2]
	\arrow["0", from=2-3, to=3-3]
	\arrow[from=3-2, to=3-3]
\end{tikzcd}
\end{equation}
By \eqref{diag:d5} we have $\cO_{X_{00}} \simeq \cO_X \otimes_{\A^1}^{L} 0 \simeq \cO_X \oplus \cO_X[1]$. Therefore, $\Qcoh(X_{00}) \simeq Mod_{\cO_X \oplus \cO_X[1]}(\Qcoh(X))$ \cite[Cor. 3.3.2]{Preygel:2011}. When $X$ is a classical scheme or stack we can realize any complex of quasi-coherent sheaves on $X_{00}$ as a co-chain complex $(E^{\bullet},d)$ of $\cO_X$-modules with a degree $-1$ chain map $h:E^{\bullet} \to E^{\bullet}$ such that $h^2=0$ and $[d,h]=0$. For more details see \cite[Remark 1.22]{pippi_HS} and \cite[Remark 2.32]{MR3877165}.
From the pullback square \eqref{diag:d5} we get a splitting $i_0':X\to X_{00}$ such that $i_0\circ i_0' = id_X$. Thus 
$${i_0}_*:\IndCoh(X_0) \to \IndCoh(X)$$ 
is essentially surjective and $i_0$ being quasi-smooth and proper ${i_0}_*$ preserves compact objects. 
Hence there exists $\sF_0 \in D^bCoh(X_{00})$ such that ${i_0}_*\sF_0 = \cO_X$.
This $\sF_0$ is the realization of $\cO_X$ as a factorization of 0, that is, as an object in $MF^{coh}(X,0).$

Now, consider the following cartesian square.
\begin{equation}\label{diag:d1}
\begin{tikzcd}[cramped]\
	{X_{00}} && {(X\times X)_0} \\
	X && {X\times X}
	\arrow["\Delta_{00}", from=1-1, to=1-3]
	\arrow["{i_0}"{description}, from=1-1, to=2-1]
	\arrow["k", from=1-3, to=2-3]
	\arrow["{\overline{\Delta}}"{description}, dashed, from=2-1, to=1-3]
	\arrow["\Delta"{description}, from=2-1, to=2-3]
\end{tikzcd}
\end{equation}
Observe that $\Delta_{00} = \ov\Delta \circ {i_0} $. Since $X$ is separated, $\Delta$ is proper. So by base change, $\Delta_{00}$ is also proper. As $(X\times X)_0$ is quasi-compact $\Delta_{00}$ has finite cohomological dimension \cite[Cor. 1.4.5]{gaitsgory}. By base change \cite[Remark 1.11 (ii)]{kth-gth-khan} applied to \eqref{diag:d1}  we have $\Delta_{00*}i_0^* = k^*\Delta_*.$
Hence, we have a $k((\beta))$-linear $\infty$-functor 
\[{\Delta_{00}}_* : MF^{coh}_{Z}(X,0) \to MF^{coh}_{Z\times Z}(X\times X,f \boxplus -f).\]

Consider the following functor:
\begin{equation}\label{eqn:d2}
    \ol\Delta_* : \IndCoh(X) \to \IndCoh((X\times X)_0).
\end{equation}
Since $k \circ \ol\Delta = \Delta$ is proper, and $k$ is separated, $\ol\Delta$ is also proper. Since $(X\times X)_0$ is quasi-compact $\ol\Delta$ has finite cohomological dimension \cite[Cor. 1.4.5]{gaitsgory}. Therefore, \eqref{eqn:d2} restricts to compact objects
\begin{equation}\label{eqn:d3}
    \ol\Delta_* : D^bCoh(X) \to D^bCoh((X\times X)_0).
\end{equation}
 By the commutativity of the diagram \eqref{diag:d1} we get 
\begin{equation}\label{eqn:factorization of ox}
    {\Delta_{00}}_* \sF_0 = \ol\Delta_* {i_0}_*\sF_0 = \ol\Delta_*\cO_X.
\end{equation}

\begin{lem}\label{lem:d1}
    Let $X$ be a derived $\DM$ stack and $Z$ be a closed sub-stack of $X$. Then $\IndCoh(X)_Z$ is generated under fiber and cofiber by the full subcategory consisting of pushforwards of objects of $\IndCoh(Z)$.
\end{lem} 

\begin{proof}
    Follows from definitions in \cite{kth-gth-khan}, \cite[Section 2.2]{devissage-algebraic-k-theory-small} and \cite[Theorem 4]{quillen}.
\end{proof}

Now, we come to the main theorem of this section.

\begin{thm}\label{thm:duality}
Let $(X,f)$ be a LG model.
Let $Z$ be a closed sub-stack of $X_0$. By Remark \ref{remk:dualizable}, $MF^{coh, \infty}_{Z}(X,f)$ is a dualizable $\infty$-category.
Then we have the following equivalence of $k((\beta))$-linear $\infty$-categories:
\[MF^{coh, \infty}_{Z}(X,f)^{\vee} \cong MF^{coh, \infty}_{Z}(X,-f).\]

\end{thm}

\begin{proof}
    The idea of the proof follows from \cite[Theorem 4.2.2]{Preygel:2011}.
    By \lemref{lem:d1}, $\IndCoh(X)_Z$ is generated 
    by pushforwards of objects of $\IndCoh(Z)$, it is therefore enough to prove the theorem without support 
    conditions. Since $MF^{coh,\infty}(X,f) = PreMF^{coh,\infty}(X,f) \otimes_{k[[\beta]]}k((\beta))$, it is enough to prove the theorem for $PreMF^{coh, \infty}(X,f)$ instead of $MF^{coh,\infty}(X,f)$.
    We will define a colimit preserving $k[[\beta]]$-linear $\infty$-functor:
    \begin{align*}
        \langle - , - \rangle^{coh, \infty} : PreMF^{coh, \infty}(X,f) \otimes_{k[[\beta]]} PreMF^{coh, \infty}(X,-f) &\to k[[\beta]]\mbox{-}\Mod 
    \end{align*} 
    and show that it is a perfect pairing  (\defref{rmk:pp}).
    Now, by Thom-Sebastiani type theorem, Theorem \ref{thm:thom-sebastiani}, we have:
    \begin{align*}
        PreMF^{coh, \infty}(X,f) \otimes_{k[[\beta]]} PreMF^{coh, \infty}(X,-f) &\cong PreMF^{coh, \infty}(X \times X, f \boxplus -f).
    \end{align*}
    Consider the $\infty$-functor:
    \begin{align*}
        \langle - , - \rangle^{coh, \infty} &: PreMF^{coh, \infty}(X,f) \otimes_{k[[\beta]]} PreMF^{coh, \infty}(X,-f) \to k[[\beta]]\mbox{-}\Mod \\
        \langle - , - \rangle^{coh, \infty} &:= RHom^{k[[\beta]]}_{PreMF^{coh, \infty}(X^2, f\boxplus -f)}(\overline{\Delta}_*\sO_{X}, \ell_*(-\boxtimes -)).
    \end{align*}
    Now, to show that $\langle -, - \rangle^{coh, \infty}$ is a perfect pairing, it is enough to consider it as a $k$-linear functor. 
    As a $k$-linear category 
    \begin{equation*}
         PreMF^{coh, \infty}(X^2, f\boxplus -f) \simeq \IndCoh((X^2)_0)
    \end{equation*}
   and 
    $$PreMF^{coh, \infty}(X, f) \simeq \IndCoh(X_0).$$
        For a fixed $\sF \in PreMF^{coh, \infty}(X,f)$, consider the $k$-linear functor 
    $$PreMF^{coh, \infty}(X,f) \to \Fun(PreMF^{coh, \infty}(X,-f), k\mbox{-}\Mod)$$ defined by 
    \begin{equation}\label{functor 1}
        \sF \mapsto RHom^{k}_{PreMF^{coh, \infty}(X^2, f\boxplus -f)}(\overline{\Delta}_*\sO_{X}, \ell_*(\sF \boxtimes -)) =^1 RHom^{k}_{\IndCoh((X^2)_0)}(\overline{\Delta}_*\sO_{X}, \ell_*(\sF \boxtimes -)).
    \end{equation}
   
    Here $=^1$ follows from the fact that as a $k$-linear category, $PreMF^{coh, \infty}(X^2, f\boxplus -f)$ is  $\IndCoh((X^2)_0)$.
    
    The functor \eqref{functor 1} is colimit preserving, as $\ol\Delta_*\sO_X$ is a compact object in $\IndCoh((X^2)_0)$, there exists a left adjoint to it. We denote it by
    $$F_{\ell} : \Fun(PreMF^{coh, \infty}(X,-f), k\mbox{-}\Mod) \simeq^{1} PreMF^{coh, \infty}(X,-f)^{op} \to PreMF^{coh, \infty}(X , f ).$$ 
    Here $\simeq^1$ follows from \cite[Lemma 4.2.1]{Preygel:2011}.
    
    It is enough to show that $F_{\ell}$ is an equivalence. 
    Now, since both domain and co-domain of the functor $F_{\ell}$ are compactly generated $\infty$-categories and $F_{\ell}$ preserves compact objects, as a $k$-linear functor $F_{\ell}$ is $\Ind$ of the functor $F_{\ell}^c$ induced from $F_{\ell}$ by restricting to compact objects,
    \[F_{\ell}^c : PreMF^{coh}(X,-f)^{op} \to PreMF^{coh}(X,f).\]
    By the following series of equivalences in \eqref{eqn:duality equiv} we can identify $F_{\ell}^c$ with the Grothendieck duality $\D(-) $ \eqref{eqn:gd} as $k$-linear functors. 
    Consider the following cartesian square:
\begin{equation}\label{eqn:duality base change}
\begin{tikzcd}
	{X_0} & X \\
	{(X_0)^2} & {(X^2)_0}
	\arrow["i", from=1-1, to=1-2]
	\arrow["{\Delta_{X_0}}"', from=1-1, to=2-1]
	\arrow["{\ol\Delta}", from=1-2, to=2-2]
	\arrow["\ell"', from=2-1, to=2-2]
\end{tikzcd}
\end{equation}
Let $\sF \in PreMF^{coh}(X,-f)$ and $ \sG \in PreMF^{coh}(X,f).$ Now, we have following equivalences:
    \begin{equation}\label{eqn:duality equiv}
        \begin{split}
            RHom_{PreMF^{coh}(X^2,f\boxplus -f)}(\ol\Delta_*\cO_X, \ell_*(\sF \boxtimes \sG))&\simeq RHom_{D^bCoh((X^2)_0)}(\ol\Delta_*\cO_X, \ell_*(\sF \boxtimes \sG))\\
            & \simeq^1 RHom_{D^bCoh((X_0)^2)}(\ell^*\ol\Delta_*\cO_X, \sF \boxtimes \sG)\\
            & \simeq^2 RHom_{D^bCoh((X_0)^2)}({\Delta_{X_0}}_*i^*\cO_X, \sF \boxtimes \sG)\\
            & \simeq^3 RHom_{D^bCoh((X_0)^2)}({\Delta_{X_0}}_*\cO_{X_0}, \sF \boxtimes \sG)\\
            & \simeq^4 RHom_{\IndCoh(X_0)}(\cO_{X_0}, {\Delta_{X_0}^!}(\sF \boxtimes \sG))\\
            & \simeq^5 RHom_{\IndCoh(X_0)}(\cO_{X_0}, R\sHom_{D^bCoh(X_0)}(\D\sF,\sG))\\ 
            & \simeq RHom_{D^bCoh(X_0)}(\D\sF,\sG)).
        \end{split}
    \end{equation}
    Here $\simeq^1$ follows from adjoint pair $\ell^* :D^bCoh((X_0)^2) \leftrightarrows D^bCoh((X^2)_0) : \ell_*$, $\simeq^2$ follows from  base change by \eqref{eqn:duality base change}, $\simeq^3$ follows from the fact that $i^*\cO_X = \cO_{X_0}$, $\simeq^4$ follows from the adjoint pair $\Delta_* : \IndCoh(X_0) \leftrightarrows \IndCoh((X_0)^2) : \Delta^!$, $\simeq^5$ follows from \cite[Proposition 4.4.4]{gaitsgory}. Since $\D$ is an equivalence \eqref{eqn:gdcoh}
    this completes the proof of the Theorem \ref{thm:duality}.
\end{proof}

\subsection{Assumption}
\begin{assumption}\label{assumption}
    For the rest of the article, we assume that our LG model $(X,f)$ has the following property:
    Consider the following homotopy cartesian squares:
\[\begin{tikzcd}[cramped]
	{X_0} & X & {X^{\times}} \\
	0 & {\A^1} & {\mathbb{G}_m}
	\arrow[from=1-1, to=1-2]
	\arrow[from=1-1, to=2-1]
	\arrow["f", from=1-2, to=2-2]
	\arrow[from=1-3, to=1-2]
	\arrow["{f^{\times}}", from=1-3, to=2-3]
	\arrow[from=2-1, to=2-2]
	\arrow[from=2-3, to=2-2]
\end{tikzcd}\]
    We assume the morphism $f^{\times}:X^{\times}\to \mathbb{G}_m$ is smooth. 
\end{assumption}

As a consequence of Assumption \ref{assumption}, we have the following.

\begin{lem}\label{lem:supp=not supp}
    Let $(X,f)$ be a LG model satisfying Assumption \ref{assumption}. Then,
    \begin{equation}
        MF^{coh}_{(X_0 \times X_0)}(X\times X, f\boxplus-f) \simeq MF^{coh}(X\times X, f\boxplus-f).
    \end{equation}
\end{lem}

\begin{proof}
    Consider the following Cartesian squares:
    \[\begin{tikzcd}[cramped]
	{X_0\times X_0} & {(X\times X)_0} & {X\times X} \\
	0 & {\A^1} & {\A^1\times \A^1} \\
	& 0 & {\A^1}
	\arrow[hook, from=1-1, to=1-2]
	\arrow[from=1-1, to=2-1]
	\arrow[hook, from=1-2, to=1-3]
	\arrow["q"', from=1-2, to=2-2]
	\arrow["{f\times f}", from=1-3, to=2-3]
	\arrow[from=2-1, to=2-2]
	\arrow["\Delta"', from=2-2, to=2-3]
	\arrow[from=2-2, to=3-2]
	\arrow["{(t_1,t_2)\mapsto t_2 -t_1}", from=2-3, to=3-3]
	\arrow[from=3-2, to=3-3]
\end{tikzcd}\]
Therefore, we have 
\[(X\times X)_0 = (X\times X) \times_{\A^1\times \A^1}\A^1 = X\times_{\A^1}X,\]
and the map $q$ is given by $X\times_{\A^1}X \to X\xrightarrow{f}\A^1$.
Let $Y=(X\times X)_0$ where $\widetilde{f}$ denotes $f\boxplus-f$. Then $(Y,q)$ is a LG model with derived zero locus $X_0\times X_0$.
Let $Y^{\times}$ be defined by the following Cartesian square,
\[\begin{tikzcd}[cramped]
	{(X\times X)_0} & {Y^{\times}} \\
	{\A^1} & {\mathbb{G}_m}
	\arrow["q"', from=1-1, to=2-1]
	\arrow[from=1-2, to=1-1]
	\arrow[from=1-2, to=2-2]
	\arrow[from=2-2, to=2-1]
\end{tikzcd}\]
Note that $Y^{\times} \simeq X^{\times} \times_{\mathbb{G}_m} X^{\times}$. Since $X^{\times}$ is smooth by our assumption, $Y^{\times} = (X\times X)_0-(X_0\times X_0)$ is also smooth. 
Now, in \cite[Lemma 3.1]{coherent-analogues} put $X\times X$ in place of $X$ and put $Z=(X\times X)_0, ~T=X_0\times X_0,~V=(X\times X)_0-(X_0\times X_0),$ and $U=(X\times X)-(X_0\times X_0)$. Then a analogous result to \cite[Lemma 3.1]{coherent-analogues} for derived $\DM$ stacks gives the following triangle:
\[MF^{coh}_{(X_0 \times X_0)}(X\times X, f\boxplus-f)\to MF^{coh}(X\times X, f\boxplus-f) \to D^b_{Sing}(V/U).\]
Now note that the inclusion map $V\hookrightarrow U$ is affine and $V$ is smooth, thus $D^b_{Sing}(V/U) =0.$
Therefore, we have 
\[MF^{coh}_{(X_0 \times X_0)}(X\times X, f\boxplus-f) \simeq MF^{coh}(X\times X, f\boxplus-f).\] This completes the proof of \lemref{lem:supp=not supp}.
\end{proof}

\subsection{Representative of evaluation and identity functors }
    Let $(X,f)$ and $(Y,g)$ be LG models.
   In this section, we identify evaluation and identity endomorphisms of the category $MF^{coh, \infty}(X,f)$ with objects in $MF^{coh, \infty}(X \times Y, -f \boxplus g)$. We need these expressions to prove the HKR theorem in the next section.
\begin{lem}\label{lem:functor-category}
	Let $(X,f)$ and $(Y,g)$ be LG models.
    Then we have the following equivalence of $k((\beta))$-linear $\infty$-categories:
    \begin{equation}\label{eqn:product}
        Fun^L_{k((\beta))}(MF^{coh, \infty}(X,f), MF^{coh, \infty}(Y,g)) \cong MF^{coh, \infty}(X \times Y, -f \boxplus g).
    \end{equation}
\end{lem}

\begin{proof}
	By Theorem \ref{thm:duality} we have $MF^{coh, \infty}(X,f)^{\vee} \cong MF^{coh, \infty}(X,-f)$.
	Now, in the category $Pr^{L,w}_{k((\beta)),\vee}$ we have the adjoint pair 
	$(Fun^L_{k((\beta))}(-,-), -\otimes_{k((\beta))}-)$ (see Notation \ref{notat:1}). Using that, we have the following 
	equivalences:
	\begin{align*}
		Fun^L_{k((\beta))}(MF^{coh, \infty}(X,f), MF^{coh, \infty}(Y,g)) &\cong Fun^L_{k((\beta))}(k((\beta)),MF^{coh, \infty}(X,-f) \otimes_{k((\beta))} MF^{coh, \infty}(Y,g))\\
		&\cong^1 Fun^L_{k((\beta))}(k((\beta)),MF^{coh, \infty}_{X_0\times Y_0}(X \times Y, -f \boxplus g))\\
        &\cong MF^{coh, \infty}_{X_0\times Y_0}(X \times Y, -f \boxplus g))\\
		&\cong^2 MF^{coh, \infty}(X \times Y, -f \boxplus g).
	\end{align*}
	The equivalence $\cong^1 $ follows from   \thmref{thm:thom-sebastiani}, and the equivalence $\cong^2$ follows from \lemref{lem:supp=not supp} applied to $(X,f)$ and $(Y,g)$.
\end{proof}

The following is the main theorem of this subsection.
\begin{thm}\label{thm: representability} \label{thm:repn of ev and id}
    Let $(X,f)$ be a LG model, then under the equivalence \eqref{eqn:product} the identity  functor corresponds to:
    \begin{align*}
        id_{MF^{coh, \infty}(X,f)} &\mapsto \overline{\Delta}_*R\Gamma_{X_0}(\omega_X),
    \end{align*}
    where $\omega_X$ is the dualizing complex of $X$ and $\overline{\Delta} : X \to (X \times X)_0$ is the induced diagonal morphism.
    The evaluation functor $$ev: MF^{coh,\infty}(X,f) \otimes_{k((\beta))}MF^{coh,\infty}(X,-f) \to k((\beta))\mbox{-}\Mod$$
    corresponds to 
    \[RHom_{MF^{coh,\infty}(X^2,f\boxplus -f)}(\ol{\Delta}_* \mathcal{O}_X,-):MF^{coh,\infty}(X^2,f\boxplus -f) \to k((\beta)\mbox{-}\Mod.\]
   
\end{thm}

\begin{proof}
    By Lemma \ref{lem:functor-category} we can view the identity functor $id_{MF^{coh, \infty}(X,f)}$ 
    as an object in $MF^{coh, \infty}(X \times X, -f \boxplus f)$.
   
    Since $MF^{coh, \infty}(X \times X, -f \boxplus f) = PreMF^{coh, \infty}(X \times X, -f \boxplus f) \otimes_{k[[\beta]]} k((\beta))$, it is enough to show that 
    \begin{equation}\label{repn identity}
        id_{PreMF^{coh, \infty}(X,f)}\mapsto\overline{\Delta}_*R\Gamma_{X_0}(\omega_X)\in \IndCoh_{(X_0)^2}((X^2)_0),
    \end{equation}
    and 
    \begin{equation}\label{repn ev}
        ev_{PreMF^{coh, \infty}(X,f)} \mapsto RHom^{k[[\beta]]}_{PreMF^{coh, \infty}(X \times X, -f \boxplus f)}(\overline{\Delta}_*\sO_X,-).
    \end{equation}

    The idea of the proof follows from \cite[Theorem 4.2.3]{Preygel:2011}. Let $\mathcal{H} = \mathcal{F}\boxtimes\mathcal{G} \in PreMF^{coh}(X^2,f\boxplus-f)$, and $T,T' \in PreMF^{coh}(X,f)$. Then we have following sequence of $k[[\beta]]$-linear isomorphisms:
    \begin{align}
RHom_{PreMF^{coh}(X,f)}(T,\Phi_{\mathcal{H}}(T'))
&\simeq RHom_{PreMF^{coh}(X,f)}(T,\Phi_{\mathcal{F}\boxtimes\mathcal{G}}(T'))\\
&\simeq^1
\begin{aligned}[t]
&RHom_{PreMF^{coh}_{(X_0\times X_0)}(X^2,f\boxplus -f)}
   (\ol\Delta_*\mathcal{O}_X,\mathcal{F}\boxtimes T') \\
&\qquad\otimes_{k[[\beta]]}
RHom_{PreMF^{coh}(X,f)}(T,\mathcal{G})
\end{aligned}\\
&\simeq^2
\begin{aligned}[t]
&RHom_{PreMF^{coh}(X^2,f\boxplus -f)}
   (\ol\Delta_*\mathcal{O}_X,\mathcal{F}\boxtimes T') \\
&\qquad\otimes_{k[[\beta]]}
RHom_{PreMF^{coh}(X,f)}(T,\mathcal{G})
\end{aligned}\\
&\simeq^3
\begin{aligned}[t]
&RHom_{PreMF^{coh}(X^2\times X,(f\boxplus -f)\boxplus f)}
   (\ol\Delta_*\mathcal{O}_X\boxtimes T,
    \mathcal{F}\boxtimes\mathcal{G}\boxtimes T')\\
\end{aligned}\\
&\simeq \label{eqn:repn1}
RHom_{PreMF^{coh}(X^2\times X,(f\boxplus -f)\boxplus f)}
(\ol\Delta_*\mathcal{O}_X\boxtimes T,
 \mathcal{H}\boxtimes T').
\end{align}
Here, $\simeq^1$ follows from \thmref{thm:thom-sebastiani} and \thmref{thm:duality}, $\simeq^2$ follows from \lemref{lem:supp=not supp}, and $\simeq^3$ follows from \thmref{thm:thom-sebastiani}.
Since $RHom(-,-)$ commutes with colimit in the second coordinate, we get, for $\mathcal{H} \in PreMF^{coh,\infty}(X^2,f\boxplus-f)$,
\[RHom_{PreMF^{coh,\infty}(X,f)}(T,\Phi_{\mathcal{H}}(T'))\simeq RHom_{PreMF^{coh,\infty}(X^2\times X,(f\boxplus -f)\boxplus f)}
(\ol\Delta_*\mathcal{O}_X\boxtimes T,
 \mathcal{H}\boxtimes T').\]
Let  $\ol\Delta_1$ and $\ol\Delta_2$ denote the following compositions,
\[X\times X_0\xrightarrow{\ol\Delta\times  id} (X\times X)_0\times X_0 \to (X^3)_0,\]and
\[X_0\times X\xrightarrow{id \times \ol\Delta} X_0 \times (X\times X)_0 \to (X^3)_0\] respectively.
Now, consider the following Cartesian square,
\begin{equation}\label{eqn:base change repn}
\begin{tikzcd}[column sep=large,row sep=large]
X_0\times \mathbb{B}^{\bullet}
  \arrow[r,"D_2"]
  \arrow[d,"D_1"']
&
X\times X_0\times \mathbb{B}^{\bullet}
  \arrow[d,"\overline{\Delta}_1"]
\\
X_0\times X\times \mathbb{B}^{\bullet}
  \arrow[r,"\overline{\Delta}_2"']
&
(X^3)_0\times \mathbb{B}^{\bullet}
\end{tikzcd}
\end{equation}
Here, $\ol\Delta_1$ and $\ol\Delta_2$ are the diagonal maps.
Now, for a fixed $\bullet$, we have the following series of equivalences.
\begin{align*}
R\!\operatorname{Hom}_{\operatorname{PreMF}^{coh,\infty}(X,f)}
\bigl(T,\Phi_{\overline{\Delta}_*R\Gamma_{X_0}\omega_X}(T')\bigr)
&\simeq^1
R\!\operatorname{Hom}_{\operatorname{PreMF}^{coh,\infty}(X^3,f\boxplus -f\boxplus f)}
\bigl(
\overline{\Delta}_*\mathcal{O}_X\boxtimes T,\,
T'\boxtimes \overline{\Delta}_*\mathrm{R}\Gamma_{X_0}\omega_X
\bigr)\\
&\simeq^2
R\!\operatorname{Hom}_{\operatorname{PreMF}^{coh,\infty}(X^3,f\boxplus -f\boxplus f)}
\bigl(
\overline{\Delta}_{1*}(\mathcal{O}_X\boxtimes T),\,
\overline{\Delta}_{2*}(T'\boxtimes \mathrm{R}\Gamma_{X_0}\omega_X)
\bigr)\\
&\simeq^3
R\!\operatorname{Hom}_{\operatorname{IndCoh}(X\times X_0)}
\bigl(
\mathcal{O}_X\boxtimes T,\,
\overline{\Delta}_1^{\,!}\,
\overline{\Delta}_{2*}
(T'\boxtimes \mathrm{R}\Gamma_{X_0}\omega_X)
\bigr)\\
&\simeq^4
R\!\operatorname{Hom}_{\operatorname{IndCoh}(X\times X_0)}
\bigl(
\mathcal{O}_X\boxtimes T,\,
(D_2)_*(D_1)^!
(T'\boxtimes \mathrm{R}\Gamma_{X_0}\omega_X)
\bigr)\\
&\simeq^5
R\!\operatorname{Hom}_{\operatorname{IndCoh}(X_0)}
\bigl(
(D_2)^*(\mathcal{O}_X\boxtimes T),\,
(D_1)^!
(T'\boxtimes \mathrm{R}\Gamma_{X_0}\omega_X)
\bigr)\\
&\simeq
R\!\operatorname{Hom}_{\operatorname{IndCoh}(X_0)}
\bigl(
\mathcal{O}_{X_0}\otimes T,\,
T'\otimes^! \mathrm{R}\Gamma_{X_0}\omega_{X_0}
\bigr)\\
&\simeq^6
R\!\operatorname{Hom}_{\operatorname{IndCoh}(X_0)}(T,T')
=
R\!\operatorname{Hom}_{\operatorname{PreMF}^{coh,\infty}(X,f)}(T,T').
\end{align*}
Here, $\simeq^1$ follows from \eqref{eqn:repn1}.
$\simeq^2$ follows from definitions of $\ol\Delta_1$ and $\ol\Delta_2$.
$\simeq^3$ follows from the adjunction $(\ol{\Delta_1}_*, \ol\Delta_1^!).$
$\simeq^4$ follows from the base change from \eqref{eqn:base change repn}.
$\simeq^5$ follows from the adjunction $(D_2^*,{D_2}_*).$
$\simeq^6$ follows from \cite[Lemma B.2.1]{Preygel:2011}.
Therefore, we have a $k[[\beta]]$-linear equivalence 
\[R\!\operatorname{Hom}^{k[[\beta]]}_{\operatorname{PreMF}^{coh,\infty}(X,f)}
\bigl(T,\Phi_{\overline{\Delta}_*R\Gamma_{X_0}\omega_X}(T')\bigr) \simeq R\!\operatorname{Hom}^{k[[\beta]]}_{\operatorname{PreMF}^{coh,\infty}(X,f)}(T,T').\]
This completes the proof of \eqref{repn identity}.
The proof of \eqref{repn ev} follows from \thmref{thm:duality}.
   \end{proof}

    \section{Hochschild homology and Hochschild cohomology of \texorpdfstring{$MF^{coh}(X,f)$}{MFrel(X,f)}}\label{sec:HKR}

In this section, we prove the HKR-type theorem for coherent matrix factorization categories. Throughout this section, fix the following notations. Let $(X,f)$ be a LG model which satisfies Assumption \ref{assumption}. Let $i: X_0 \to X$ be the derived zero locus of $f$. We will show that the Hochschild complex of the coherent matrix factorization category, $MF^{coh}(X,f)$, is isomorphic to the hyper-cohomology of the twisted Hodge complex of $I^{DM}X$, where the differential is given by $-\wedge df\mid_{I^{DM}X}.$

\subsection{HKR isomorphism for Hochschild homology}
\begin{thm}\label{thm: HKR homo}
	Let $X$ be a quasi-smooth, finite type, 
 separated, derived $\DM$ stack over a field $k$ of char 0 with affine stabilizers.
Let $(X,f)$ be a LG model satisfying Assumption \ref{assumption}. 
Then we have the following equivalence of $k((\beta))$-modules:
\begin{equation}\label{hkr isom mf}
    HH_{\bullet}(MF^{\mathrm{coh}}(X,f)) \simeq R\Gamma(I^{DM}X, Sym(\mathbb{L}_{I^{DM}X}[1]))^{B{\mathbb{G}_a}}\otimes_{k[[\beta]]} k((\beta)).
\end{equation}
\end{thm}
Before proving the theorem, we prove the following lemma which gives an alternative characterization of $MF^{coh, \infty}(X,f)$ that will be useful in the proof of the HKR type isomorphism in Theorem \ref{thm: HKR homo}. 
Recall the definition of homotopy orbits and homotopy fixed points from Section \ref{sec:homotopical alg}.

\begin{lem}\label{lem:hocolim}
    Let $\sF \in \Qcoh(X_0) $. Then there exists a functorial equivalence
    \[\hocolim \{  \cdots \to i^*i_*\sF[2] \to i^*i_*\sF[1] \to i^*i_*\sF \} \xrightarrow{\simeq} \sF,\]
    which induces the following equivalence:
    \[RHom^{k[[\beta]]}_{\Qcoh(X_0)}(\sF,\sG) \simeq RHom^{k}_{\Qcoh(X)}(i_*\sF,i_*\sG)^{S^1}.\]
\end{lem}

\begin{proof}
    Let $\sF,\sG \in \Qcoh(X_0)$ and let $i: X_0 \to X$  denote the natural quasi-smooth closed immersion, where $X_0$ is the derived zero locus of $f$. 
    Now, $i^*i_* \in \Fun^L_k(\Qcoh(X_0), \Qcoh(X_0))$. By \cite[Example 3.1.11]{Preygel:2011} we have the following equivalence of endofunctors on $\Qcoh(X_0)$,
    \[(i^*i_*)_{S^1} \simeq id_{\Qcoh(X_0)}.\]
    Applying \defref{fixed pt and orbit} (ii) for $G=S^1$ we get,
    \begin{equation}\label{eqn:hocolim1}
        (i^*i_*)_{S^1}(\sF) \simeq \hocolim_{BS^1}\{i^*i_*\sF\}.
    \end{equation}
    To calculate this colimit, one can take the simplicial bar resolution of $BS^1$ defined as $[n]\mapsto (S^1)^n$ with usual face and degeneracy maps. By the Dold-Kan correspondence, we can take the colimit over the normalized chain complex. (see \cite[Example 3.1.11]{Preygel:2011} for details). Hence we get
    \begin{equation}\label{eqn:hocolim}
        (i^*i_*)_{S^1}(\sF) \simeq \hocolim\{\cdots \to (i^*i_*)\sF[2] \to (i^*i_*)\sF[1] \to (i^*i_*)\sF\}.
    \end{equation}
    Therefore, we get the first claim of Lemma \ref{lem:hocolim}.

    The proof of the second part goes along the same lines as in \cite[Proposition 3.2.1]{Preygel:2011}. We need to work in $\Qcoh(X_0)$ instead of $D^bCoh(X_0)$. Consider the following equivalences:
    \begin{align*}
        RHom^{k[[\beta]]}_{\Qcoh(X_0)}(\sF,\sG) &{=}^{a} RHom^{k[[\beta]]}_{\Qcoh(X_0)}(\hocolim\{\cdots \to (i^*i_*)\sF[2] \to (i^*i_*)\sF[1] \to (i^*i_*)\sF\},\sG)\\
        &{=}^{b} \holim_{BS^1} RHom^{k[[\beta]]}_{\Qcoh(X_0)}(i^*i_*\sF,\sG)\\
        &{=}^{c} \holim_{BS^1} RHom^{k}_{\Qcoh(X)}(i_*\sF,i_*\sG) \\
        &{=}^{d} RHom^{k}_{\Qcoh(X)}(i_*\sF,i_*\sG)^{S^1}.
    \end{align*}

    Here, ${=}^{a}$ follows from \eqref{eqn:hocolim}, ${=}^{b}$ follows as $Hom(-,\sG)$ is a contravariant left exact functor. We are in a presentable stable $\infty$-category $\Qcoh(X_0)$ so it is exact \cite[Proposition 1.1.4.1]{HA}, ${=}^{c}$ follows from adjoint pair $i^* : \Qcoh(X) \leftrightarrows \Qcoh(X_0) : i_*$, and ${=}^{d}$ follows from the definition of homotopy fixed points \defref{fixed pt and orbit}.
    This completes the proof.
\end{proof}

Now, we prove a standard known fact in \lemref{clm:q^!} which will be used to prove the HKR isomorphism. We will briefly recall the determinant ($\det$) of a perfect complex from \cite[Section 3]{Toen_Atiyah-class}; a graded version of this is also discussed in \cite[Appendix B]{hlp-theta-strata}.
For a derived stack $X$, let $(Vect(X)^{\simeq}, \oplus)$ denote the symmetric monoidal $\infty$-groupoid of locally free sheaves and isomorphisms with symmetric monoidal structure given by direct sum and $\Pic (X)$ denote the group like $\E_{\infty}$ monoid of invertible objects and isomorphisms between them. Let $K(\sC)$ denote the algebraic $K$-theory of a category $\sC$ , for details see \cite[Section 2]{kth-gth-khan}, \cite[Section 7.1]{k-th-tabuada}. Then 
\begin{equation}
    \det (E) := \bigwedge^{\rank(E)} (E)
\end{equation}
defines a symmetric monoidal functor of symmetric monoidal $\infty$-groupoids
\begin{equation} \label{det1}
    \det : (Vect(X)^{\simeq}, \oplus) \to (\Pic (X), \otimes).
\end{equation}
Then following \cite[pg 87]{hlp-theta-strata} one constructs
\begin{equation}
    u: (Vect(X),\oplus)^{gp} \to K(Perf(X)) 
\end{equation}
where $(-)^{gp}$ denotes the group-like completion of a symmetric monoidal $\infty$-monoid. Since $\Pic(X)$ is group like there exists an unique extension of $\det$ in \eqref{det1} to 
\begin{equation}
    \det : (Vect(X),\oplus)^{gp} \to \Pic (X).
\end{equation}
Now, $u$ and $\det$ are functorial in $X$, so one can think of them as presheaves valued in group-like symmetric monoidal $\infty$-groupoids on $X$. By \cite[Theorem 1.19]{det-heleodoro}, $u$ is an isomorphism for affine derived schemes. Hence, the shefified functor $u^{sh}$ is an isomorphism of sheaves. Note that $\Pic$ is a smooth (hence \'etale) sheaf. Therefore, define the $\det$ functor as the following composition \cite[pg. 87]{hlp-theta-strata},
\begin{equation}
    K(Perf(X)) \to \Gamma(X,K(Perf(-))^{sh}) \xrightarrow{\det \circ (u^{sh})^{-1}} \Gamma(X,\Pic(-)) \simeq \Pic(X).
\end{equation}
Some key properties of the $\det$ are,
\begin{enumerate}
    \item[$\bullet$] Let $W\xrightarrow{f} X \xrightarrow{g} Y$ be morphisms between quasi-smooth derived algebraic stacks with perfect cotangent complexes $\L_f$ and $\L_g$. Then $\det$ applied to the cofiber sequence $f^*\L_g \to \L_{g\circ f} \to \L_f$ induces a canonical isomorphism \cite[eqn (46)]{hlp-theta-strata}
    \begin{equation}\label{det2}
        \det(\L_{g\circ f}) \simeq \det(\L_f) \otimes f^*\det(\L_g).
    \end{equation}
    \item[$\bullet$] For any morphism $P\to Q \in Perf(X)$ there exists an isomorphism \cite[pg.88]{hlp-theta-strata}
    \begin{equation}
        \det (Q) \simeq \det(P) \otimes \det(\hocofib(P\to Q)).
    \end{equation}
\end{enumerate}

\begin{lem}\label{clm:q^!}
    Consider the following composition of maps
    \[T[-1]I^{DM}X \xrightarrow{\pi} I^{DM}X \xrightarrow{p} X\]
    and let $q=p\circ \pi$. Then, 
    $$q^!\omega_X \simeq \cO_{T[-1]I^{DM}X}.$$
\end{lem}

\begin{proof}
To see the claim, consider the fundamental exact triangle of cotangent complexes \cite[Theorem 2.21 (i)]{khan_derived-geom} corresponding to the composition of maps
\[T[-1]I^{DM}X \xrightarrow{\pi} I^{DM}X \to \Spec k,\] which is 
\begin{equation}\label{eqn:tangent det}
    \pi^*\L_{I^{DM}X} \to \L_{T[-1]I^{DM}X} \to \L_{\pi}.
\end{equation}
By \cite[Remark 2.1]{Shifted_cotangent} we have $\L_{\pi} \simeq \pi^*\L_{I^{DM}X}[1]$. Therefore, applying $\det$ to \eqref{eqn:tangent det} we get,
\begin{align*}
    q^!\omega_X = \omega_{T[-1]I^{DM}X}&=\det(\L_{T[-1]I^{DM}X}) \\
    &=^1\det(\pi^*\L_{I^{DM}X}) \otimes \det(\pi^*\L_{I^{DM}X}[1])\\
    & =^2 \pi^* (\det(\L_{I^{DM}X}) \otimes \det(\L_{I^{DM}X}[1]))\\
    & =^3 \pi^* (\det(\L_{I^{DM}X}) \otimes \det(\L_{I^{DM}X})^{-1})\\
     & = \pi^* (\cO_{I^{DM}X})\\
    & = \cO_{T[-1]I^{DM}X}.
\end{align*}
Here $=^1$ follows from \eqref{det2}, $=^2$ follows from the fact that $\pi^*$ is monoidal, and
$=^3$ follows by applying \eqref{det2} to the following triangle,
\[\cO_{I^{DM}X} \to 0 \to \cO_{I^{DM}X}[1] \xrightarrow{+1}.\]
This completes the proof.
\end{proof}

We now prove the HKR-type theorem for coherent matrix factorization categories.
\begin{proof}[Proof of \thmref{thm: HKR homo}]
We have the following series of isomorphisms:
\begin{align*}
HH_{\bullet}(MF^{\mathrm{coh}}(X,f)) 
&  \simeq^1  ev(\mathrm{id}_{MF^{\mathrm{coh},\infty}(X,f)}) \\
&  \simeq^2 RHom^{k((\beta))}_{MF^{\mathrm{coh},\infty}(X^2, f\boxplus -f)}
\big(\overline{\Delta}_*\mathcal{O}_X,\ \overline{\Delta}_*\omega_X \big) \\
&\simeq^3 RHom^{k}_{\Qcoh(X^2)}
\big(\Delta_*\mathcal{O}_X,\ \Delta_*\omega_X \big)^{S^1} \otimes_{k[[\beta]]} k((\beta)) \\
&  \simeq^4 RHom^{k}_{\Qcoh(X)}
\big(\Delta^*\Delta_*\mathcal{O}_X,\ \omega_X \big)^{S^1} \otimes_{k[[\beta]]} k((\beta))
\end{align*}
$\simeq^1$ follows from definition of Hochschild homology for a $k((\beta))$-linear $\infty$-category; see Definition~\ref{defn:HH}.
$\simeq^2$ follows from representability of the identity and evaluation functors by kernels; see Theorem~\ref{thm: representability}.
$\simeq^3$ follows from application of Lemma~\ref{lem:hocolim}, together with the fact that $\overline{\Delta_*}\mathcal{O}_X$ and $\overline{\Delta_*}\omega_X$ lie in the essential image of $\Qcoh(X^2) \hookrightarrow \IndCoh((X^2)_0)$.
$\simeq^4$  follows from adjunction $(\Delta^*,\Delta_*)$ on $\Qcoh(X^2)$.

Continuing from above, we have the following equivalences:
\begin{align*}
RHom^{k}_{\Qcoh(X)}
\big(\Delta^*\Delta_*\mathcal{O}_X,\ \omega_X \big)^{S^1} \otimes_{k[[\beta]]} k((\beta))
&  \simeq^5 RHom^{k}_{\Qcoh(X)}
\big(p_*p^*\mathcal{O}_X,\ \omega_X \big)^{S^1} \otimes_{k[[\beta]]} k((\beta))\\
&  \simeq^6 RHom^{k}_{\Qcoh(X)}
\big(p_*\mathcal{O}_{\mathcal{L}X},\ \omega_X \big)^{S^1} \otimes_{k[[\beta]]} k((\beta))\\
&  \simeq^7 RHom^{k}_{\Qcoh(\mathcal{L} X)}
\big(\mathcal{O}_{\mathcal{L}X}, p^!\omega_X \big)^{S^1} \otimes_{k[[\beta]]} k((\beta))
\end{align*}
$\simeq^5$ follows from identification $\Delta^*\Delta_* \simeq p_*p^*$ induced by the Cartesian diagram \eqref{eqn:loop stack} defining the loop stack $\mathcal{L}X := X \times_{X \times X} X$.
$\simeq^6$ follows from the fact $p^*\mathcal{O}_X \simeq \mathcal{O}_{\mathcal{L}X}$ for the structure morphism $p:\mathcal{L}X \to X$.
$\simeq^7$ follows from adjunction $(p_*,p^!)$ in $\Qcoh$ \cite[Section 3.3]{indcoherentsheaves}. Continuing from above, we have the following equivalences:
\begin{align*}
RHom^{k}_{\Qcoh(\mathcal{L} X)}
\big(\mathcal{O}_{\mathcal{L}X}, p^!\omega_X \big)^{S^1} \otimes_{k[[\beta]]} k((\beta))&  \simeq^{8} R\Gamma(\mathcal{L}X, p^!\omega_X)^{S^1} \otimes_{k[[\beta]]} k((\beta)) \\
&  \simeq^{9} R\Gamma\big(T[-1]I^{DM}X, 
q^!\omega_X\big)^{B{\mathbb{G}_a}}\otimes_{k[[\beta]]} k((\beta)) \\ 
&  \simeq^{10} R\Gamma\big(T[-1]I^{DM}X, 
\mathcal{O}_{T[-1]I^{DM}X}\big)^{B{\mathbb{G}_a}}\otimes_{k[[\beta]]} k((\beta)) \\
&  \simeq^{11} R\Gamma(I^{DM}X, Sym(\mathbb{L}_{I^{DM}X}[1]))^{B{\mathbb{G}_a}} \otimes_{k[[\beta]]} k((\beta)) .
\end{align*}
$\simeq^{8}$ follows from identification of $R\Hom_{\Qcoh(\mathcal{L} X)}(\cO_{\mathcal{L}X},- )$ with derived global sections $R\Gamma(\mathcal{L}X,-)$.
$\simeq^{9}$ follows from application of the  (HKR) isomorphism \thmref{HKR-fu+}, identifying functions on $\mathcal{L}X$ with functions on $T[-1]I^{DM}X$, compatible with the $S^1$-action by Theorem \ref{HKR-fu+}.
$\simeq^{10}$ follows from description of $q^!\mathcal{O}_X$ via the dualizing complex; see Lemma~\ref{clm:q^!}.
$\simeq^{11}$ follows from expression of functions on the shifted tangent stack $T[-1]I^{DM}X$ as the symmetric algebra on the shifted cotangent complex, \defref{defn:tangent stack}.
This completes the proof.
\end{proof}

\begin{remk}\label{two actions}
    In the isomorphism \eqref{hkr isom mf}, we have to deal with two $S^1$ actions on $\mathcal{L}X$; one is loop rotation, and the other is induced from the potential of the LG model $(X,f)$. By the affinization map Theorem \ref{HKR-fu+}, this induces two different actions of $B\G_a$ on $T[-1]I^{DM}X$. The $B\G_a$ action corresponding to loop rotation is discussed in Section \ref{sec:loop rotation and BG-a action} and the $B\G_a$ action corresponding to the potential is discussed in Section \ref{sec:BG-a action induced from potential}.
\end{remk}

\begin{remk}\label{BG-a with endo}
    Note that $B\G_a = Aff(S^1) = \Spec (k \oplus k[1])$. Therefore, $B\G_a$-action on $\cO_{T[-1]I^{DM}X}$ is determined by a degree $-1$ differential on $\cO_{T[-1]I^{DM}X}$.
\end{remk}

\subsection{\texorpdfstring{Loop rotation and \(B\mathbb{G}_a\)-action}{Loop rotation and BGa-action}}
\label{sec:loop rotation and BG-a action}

Consider the natural $S^1$ action on $\mathcal{L}X$ by rotation of loops. By HKR isomorphism of $\mathcal{L}X \simeq T[-1]I^{DM}X$ this induces a $B\G_a$ action on $T[-1]I^{DM}X$. First, we recall the definition of the derived de Rham complex of a derived stack $X$ from \cite[Section 5.3]{raksit-hochschildhomologyderivedrham}, \cite[Section 5]{HKR2026}, \cite[Section 5.2]{bms}, \cite[Section 2.2]{marangoni:tel-02957674}.
Let $\L_X$ be the cotangent complex of $X$ \cite[Section 3.2]{Thesis_Lurie}. For $i>0$ define the \emph{i-th derived exterior power} \cite[Notation 5.3.5]{raksit-hochschildhomologyderivedrham} of $\L_X$ over $\cO_X$ by
\[\bigwedge^i_{\cO_X}\L_X := \Sym^i_{\cO_X}(\L_X[1])[-i].\]
Since $\L_X$ and $\cO_X$ are connected by \cite[Proposition 25.2.4.2]{sag}, this agrees with the left derived functor of the classical exterior power functor. In particular if $X$ is smooth then $\bigwedge^i_{\cO_X}\L_X \simeq \bigwedge^i_{\cO_X}\Omega_X.$
\begin{defn}\cite[eqn. 5.2.4]{HKR2026}
    For a derived $\DM$ stack $X$, we define the derived de Rham complex as the totalization of the following mixed graded algebra:
\begin{equation}
    DR(X) := Tot \left[ \cO_X \xrightarrow{d_{dR}} \L_X \xrightarrow{d_{dR}}  \bigwedge^2_{\cO_X}\L_X \xrightarrow{d_{dR}} \cdots \right],
\end{equation}\label{eqn de rham}
where $\cO_X$ is in cohomological degree 0 and $d_{dR}$ is the de Rham differential.
\end{defn}
The derived de Rham complex has a natural filtration called the \emph{Hodge filtration} \cite[eqn. 5.2.5]{HKR2026} which is defined as: for $i>0$ the $i$-th filtered part is given by
\begin{equation}\label{eqn: hodge fil}
    Fil^i_{Hdg}DR(X) := Tot \left[ \cdots 0\to   \bigwedge^{i}_{\cO_X}\L_X \xrightarrow{d_{dR}}  \bigwedge^{i+1}_{\cO_X}\L_X \xrightarrow{d_{dR}} \cdots \right],
\end{equation}
and $Fil^i_{Hdg}DR(X) = DR(X)$ for $i\leq 0$.

 Now, similar to \cite[Proposition 4.7]{Ben_Zvi_2012} we have the following;
\begin{thm}\label{lem:de Rham action}
    Consider the action map of $B\G_a$ on $T[-1]I^{DM}X$ induced from loop rotation on $\mathcal{L}X$ via \eqref{hkr isom mf} (see Remark \ref{two actions})
    \[B\G_a \times T[-1]I^{DM}X \xrightarrow{\mu} T[-1]I^{DM}X.\]
    Its evaluation on the pullback of functions along the natural projection $T[-1]I^{DM}X \to X$ is given by the de Rham differential on the derived de Rham complex
    \[\cO_{I^{DM}X} \xrightarrow{d_{dR}} \L_{I^{DM}X}.\]
    (Explanation: the action $\mu$ determines a degree -1, square zero endomorphism $d$ of $\cO_{T[-1]I^{DM}X}$ (see Remark \ref{BG-a with endo}). Now, $\cO_{T[-1]I^{DM}X} = \Sym(\L_{I^{DM}X}[1]) = DR(I^{DM}X).$  This lemma guarantees that $d=d_{dR}$.)
\end{thm}

\begin{proof} The proof is similar to \cite[Proposition 4.7]{Ben_Zvi_2012}. We will suitably modify each step to incorporate the inertia stack. With the same notations as in \defref{defn:orbifold inertia} let $\wh{S^1}:= \lim_r BC_r \in Pro(DSt_{/k}).$
    Recall that 
    \begin{equation}\label{adjoint 1}
        \cO:DSt_{/k} \leftrightarrows DGA^{op}_k : \Spec
    \end{equation} is an adjoint pair \cite[Proposition 3.1]{Ben_Zvi_2012} and 
    \begin{equation}\label{adjoint 2}
        \Map(B\G_a \times \wh{S^1},-) : DSt_{/k} \leftrightarrows DSt_{/k} : (B\G_a \times \wh{S^1}) \times -
    \end{equation} 
    is also an adjoint pair \cite[Lemma 3.2]{noohi}. 
    Consider the following two functors:
    \[\ell : DSt_{/k} \to DGA_k^{op}\] 
    defined by,
    \[X \mapsto (\cO(B\G_a \times \wh{S^1} \times X));\]
    and
    \[r : DGA^{op}_{k} \to DSt_{/k}\]
    defined by,
    \[R \mapsto \Map(B\G_a \times \wh{S^1}, \Spec(R)) .\]
     Thus by \eqref{adjoint 1} and \eqref{adjoint 2} $\ell$ is the left adjoint of $r$.

    Let $\ell_{aff}$ denote the restriction of $\ell$ to the full subcategory of affine derived stacks (see Section \ref{sec:affine stack})  $AffDSt_{/k}$ of $DSt_{/k}$ 
    \begin{equation}\label{eqn:l_aff}
        \ell_{aff} : AffDSt_{/k} \to DGA_{k}^{op}
    \end{equation} given by
    \[\Spec R \mapsto \cO({\wh{S^1}\times \Spec R}) \oplus \cO({\wh{S^1}\times \Spec R})[-1]. \]
    Then the right adjoint $r=r_{aff}$ has the following specific form 
    \begin{equation}\label{eqn:r_aff}
        r_{aff}(R) =  \left( \Spec(\Sym ~\L_{I^{DM}\Spec R}[1])  \right),
    \end{equation}
    because of the equivalence \cite[Section 4]{HKR2026},
    \[T[-1]I^{DM}X \simeq \Map(B\G_a \times \wh S^1, X).\]
   
    and the differential is given by the de Rham differential.
       
    To conclude the proof of the proposition, we have to calculate the counit of these adjunctions
    \[\ell \circ r (\cO(X)) \to \cO(X).\]
    This is given by 
    \begin{equation}\label{eqn:dl1}
        \cO(X) \to \cO(B\G_a \times \wh{S^1} \times \Map(B\G_a \times \wh{S^1} , X)).
    \end{equation}
    Using the isomorphism \cite[Definition 3.1]{HKR2026},
    \begin{equation}\label{eqn:dl0}
        \Map(\wh S^1 , X) \simeq I^{DM}X
    \end{equation}
    the right hand side of \eqref{eqn:dl1} is isomorphic to
    \[\cO(B\G_a \times \wh S^1 \times \Map(B\G_a,I^{DM}X)).\]
    By the definition of $B\G_a$, we get that this is further isomorphic to 
    \[\cO(\wh S^1 \times \Map(B\G_a, I^{DM}X)) \oplus \cO(\wh S^1 \times \Map(B\G_a, I^{DM}X))[-1].\]
    Therefore, \eqref{eqn:dl1} induces the following morphism to the square zero extension by the shifted structure sheaf of $\Map(B\G_a, I^{DM}X)$,
    \begin{equation}
        \cO({I^{DM}X}) \to \cO(  \Map(B\G_a, I^{DM}X)) \oplus \cO(  \Map(B\G_a, I^{DM}X))[-1].
    \end{equation}
    By construction, the projection to the first factor comes from the natural projection $$T[-1]I^{DM}X \simeq \Map(B\G_a, I^{DM}X) \to I^{DM}X.$$ To prove the lemma, we have to identify the second projection to the de Rham differential.

    Now, if we assume $X$ is an affine derived stack, considering the counit morphism of the restricted functors to $AffDSt_{/k}$, we get,
    \[l_{aff}\circ r_{aff}(\cO(X)) \to \cO(X).\]
    Applying definitions \eqref{eqn:l_aff} and \eqref{eqn:r_aff} we get the following morphism in $DGA_k$,
    \begin{equation}\label{eqn:dl2}
        \cO(X) \to \cO({\wh S^1 \times \Spec(\Sym(\L_{I^{DM}X}[1]))}) \oplus \cO({\wh S^1 \times \Spec(\Sym(\L_{I^{DM}X}[1]))}[-1])
    \end{equation}
    \begin{equation}
        \cO({\Spec I^{DM}X}) \to \cO({ \Spec(\Sym(\L_{I^{DM}X}[1]))}) \oplus \cO({\Spec(\Sym(\L_{I^{DM}X}[1]))}[-1]).
    \end{equation}
    The projection to the first factor is induced from the natural projection $\Spec(\Sym(\L_{I^{DM}X}[1])) \to \Spec (I^{DM}X).$
    The projection to the second factor is given by
    \begin{equation}
        \cO({I^{DM}X}) \xrightarrow{d_{dR}} \L_{I^{DM}X} \xrightarrow{r} \Sym(\L_{I^{DM}X}[1])[-1],
    \end{equation}
    where $r$ is the inclusion map. Comparing two right adjoints, we get the following commutative diagram,
    \[\begin{tikzcd}[cramped]
	       {\cO({I^{DM}X}}) && {\cO({(  \Map(B\G_a, I^{DM}X))}[-1]}) \\
	       {\cO({I^{DM}X}}) && { \Sym(\L_{I^{DM}X}[1])[-1]}
	       \arrow[from=1-1, to=1-3]
	       \arrow["id"', from=1-1, to=2-1]
	       \arrow["h", from=1-3, to=2-3]
	       \arrow["{d_{dR} \circ r}"{description}, from=2-1, to=2-3].
    \end{tikzcd}\]
    Here vertical maps are isomorphisms, $h$ is coming from the equivalence \cite[Proposition 4.8]{HKR2026} 
    $$T[-1]I^{DM}X \simeq \Map(B\G_a, I^{DM}X).$$ 
    Comparison of the horizontal morphisms concludes the proof.
\end{proof}

\subsection{Homotopy fixed points and twisted differential}\label{sec:BG-a action induced from potential}
Let $(X,f)$ be a LG model. Let $X_{00}\xrightarrow{i_0}X$ denotes the derived zero locus of $0:X\to \A^1,$ and let $(X\times X)_{0} \xrightarrow{k} X\times X$ denotes the derived zero locus of $f\boxplus -f : X\times X \to \A^1$ (see \eqref{eqn:f box g}). 
\begin{equation}\label{eqn:pullback1}
\begin{tikzcd}[cramped]
	{X_{00}} && {(X\times X)_0} \\
	X && {X\times X}
	\arrow["{\Delta_{00}}", from=1-1, to=1-3]
	\arrow["{i_0}"', from=1-1, to=2-1]
	\arrow["k", from=1-3, to=2-3]
	\arrow["{\overline\Delta}"{description}, from=2-1, to=1-3]
	\arrow["\Delta"', from=2-1, to=2-3]
\end{tikzcd}
\end{equation}
Let $\Lambda : = H_*(S^1;k) = k[\epsilon]/\epsilon^2$ with $\epsilon$ is a variable of degree 1. By Lemma   \ref{lem:hocolim} or \cite[Lemma 3.5]{HLP_hodge} we get the $\epsilon$ action on 
\begin{equation}\label{eqn: rhom1}
    RHom_{\Qcoh(X\times X)}(k_*\ov\Delta_*\cO_X,k_*\ov\Delta_*\cO_X)
\end{equation}
is given by
\[\phi \mapsto \epsilon_X \circ \phi - (-1)^{|\phi|}\phi \circ \epsilon_X,\]
where $\epsilon_X$ is the variable in $\cO_{({X\times X})_0} := \{\cO_X[\epsilon_X]/\epsilon_X^2=0,d\epsilon=f\boxplus -f\}.$ By adjunction $(k^*,k_*)$   \eqref{eqn: rhom1} is isomorphic to 
\begin{equation}\label{eqn: rhom2}
    RHom_{\Qcoh((X\times X)_0)}(k^*k_*\ov\Delta_*\cO_X, \ov\Delta_*\cO_X).
\end{equation}
The $\Lambda$-module structure on $\eqref{eqn: rhom2}$ is given by
\[\epsilon : \phi(-) \mapsto -\phi(\epsilon\cdot (-) ).\]
Therefore, we have a $\Lambda$ module structure on $k^*k_*\ov\Delta_*\cO_X$. 
Hence, we have an induced action of $B\G_a$ on $\cO_{T[-1]I^{DM}X}$ by the isomorphisms in Theorem \ref{thm: HKR homo}, which depends on the potential of the LG model $(X,f)$. Via the isomorphism in \thmref{HKR-fu+}, we have the following characterization of this action, which is a generalization to the non-smooth setting of the result of  Preygel in \cite[Proposition 8.2.4]{Preygel:2011}.

\begin{thm}\label{lem:dr twist}
    Consider the action of $B\G_a$ on $\cO_{T[-1]I^{DM}X} =_{({\rm by}~{\rm Lemma~}\ref{lem:de Rham action})} (\Sym(\L_{I^{DM}X}[1]), d_{dR})$ induced from the potential.
    The $B\G_a$ action induced on the  derived de Rham complex $(\Sym(\L_{I^{DM}X}[1]), d_{dR})$ from the potential is equivalent to  $d_{dR}f\mid_{I^{DM}X}\wedge -.$
\end{thm}

\begin{proof}
    
    By affinization, $B\G_a = \Spec (\cO(S^1))$ \eqref{affinization of BG-a}. The $B\G_a$-action on $\Sym(\L_{I^{DM}X}[1])$ is equivalent to a degree $-1$ square zero endomorphism of $(\Sym(\L_{I^{DM}X}[1]), d_{dR})$ by \remref{BG-a with endo}. Let $\Lambda = \cO(S^1) = k[\epsilon]/\epsilon^2$ where $\epsilon$ is a variable of degree $-1$. 
    
    Therefore, to prove the lemma, it is enough to show that the action of $\epsilon$ on $\Sym(\L_{I^{DM}X}[1])$ is given by $-\wedge df\mid_{I^{DM}X}$ where $d=d_{dR}$. 
    Now, By the isomorphism of Theorem \ref{HKR-fu+} we have $\mathcal{L}X \simeq T[-1]I^{DM}X$ over $X$. Taking corresponding ring label maps, we get the following commutative diagram of $\cO(X)$-algebra homomorphisms (for a derived stack $Y$, we denote $\cO_Y(Y)$ by $\cO(Y)$ in the following)
  
\[\begin{tikzcd}[cramped]
	{\cO(\Map(S^1,X))} && {\cO(\Map(B\G_a,I^{DM}X))} \\
	\\
	{k[x]} & {\cO(X)}
	\arrow["\simeq"', from=1-1, to=1-3]
	\arrow["{x\mapsto f}"{description}, from=3-1, to=1-1]
	\arrow["{x\mapsto df\mid_{I^{DM}X}}"{description}, from=3-1, to=1-3]
	\arrow["{x\mapsto f}"', from=3-1, to=3-2]
	\arrow["\alpha"{description}, from=3-2, to=1-1]
	\arrow["{\alpha'}"{description}, from=3-2, to=1-3]
\end{tikzcd}\]
where $\alpha$ and $\alpha'$ are structure maps induces from the natural projections $\mathcal{L}X \to X$ and $T[-1]I^{DM}X \to X$ respectively.
Here, multiplication on the left-hand side corresponds to the wedge product on the right-hand side.
Therefore, using multiplicativity of HKR isomorphism \cite[Proposition 5.6]{HKR2026}, \cite[Theorem 4.1]{TV_multi_hkr} we get the $\epsilon$-action on $\cO_{T[-1]{I^{DM}X}}$ is given  by $-\wedge df\mid_{I^{DM}X}.$ This completes the proof.
\end{proof}

\begin{cor}\label{cor:homo}
    Let $(X,f)$ be a LG model. Then we have the following equivalence of  2-periodic complexes,
    \begin{equation}\label{f3}
        HH_{\bullet}(MF^{\mathrm{coh}}(X,f)) \xrightarrow{\simeq} R\Gamma \left(\Tot\left( \cdots \to \bigoplus_{i~ {\rm even}}\bigwedge^i\L_{I^{DM}X} \xrightarrow{(-)\wedge df\mid_{I^{DM}X}} \bigoplus_{i~ {\rm odd}}\bigwedge^i\L_{I^{DM}X} \cdots \right) \right).
    \end{equation}
\end{cor}

\begin{proof}
By \thmref{thm: HKR homo}, it suffices to compute $(\Sym(\L_{I^{DM}X}[1]))^{B\G_a}$, where the $B\G_a$-action is given by $-\wedge df\mid_{I^{DM}X}$, as described in \thmref{lem:dr twist} where $d$ is the de Rham differential. 

Let $\mathfrak{D} := (\Sym(\L_{I^{DM}X}[1]),0)$. Then we have,
\begin{equation}\label{eqn:f1}
    \mathfrak{D}^{B\G_a} \simeq_{by~~\eqref{eqn: fixpt adjunction}} Map_{\Lambda\mbox{-}dg\mbox{-}\Mod}(k, \mathfrak{D}).
\end{equation}
Now, by \cite[eqn. 2.3.49]{MR3877165} the RHS of \eqref{eqn:f1} is isomorphic to 
\[(\mathfrak{D}\otimes_k k[\beta], (-)\wedge df\mid_{I^{DM}X})\]
where $\beta$ is a variable of degree $2$.
Calculating this tensor product, we get the right-hand side of \eqref{f3}.

\end{proof}

\subsection{HKR type theorem for Hochschild cohomology}

In this section, we prove HKR-type of isomorphism for the Hochschild cohomology of the coherent matrix factorization category. The main theorem is as follows.
\begin{thm}\label{thm:Hochs cohomo main}
	Let $X$ be a quasi-smooth, finite type, 
 separated, derived $\DM$ stack over a field $k$ of characteristic 0 with affine stabilizers.
Let $f: X \to \mathbb{A}^1$ be a morphism, such that the LG model $(X,f)$ satisfies Assumption \ref{assumption}. 
Then we have the following equivalence of $k((\beta))$-modules:
\begin{equation*}
	HH^{\bullet}(MF^{\mathrm{coh}}(X,f)) \simeq R\Gamma(I^{DM}X, \Sym(\mathbb{L}_{I^{DM}X}[1]) \otimes \pi^*\omega_X^{\vee}[-\dim X])^{B{\mathbb{G}_a}}\otimes_{k[[\beta]]} k((\beta)).
\end{equation*}
where $\pi : I^{DM}X \to X$ is the natural projection and $\omega_X$ is the dualizing complex of $X$.
\end{thm}

\begin{proof}
	We have the following series of isomorphisms:
\begin{align*}
HH^{\bullet}(MF^{\mathrm{coh}}(X,f)) 
&\simeq^0 RHom^{k((\beta))}_{MF^{\mathrm{coh},\infty}(X^2, f\boxplus -f)}
\big(\overline{\Delta}_*\omega_X,\ \overline{\Delta}_*\omega_X \big) \\
&\simeq^1 RHom^{k((\beta))}_{MF^{\mathrm{coh},\infty}(X^2, f\boxplus -f)}
\big(\overline{\Delta}_*\mathcal{O}_X,\ \overline{\Delta}_*\mathcal{O}_X \big) \\
&\simeq^2 RHom^{k}_{\Qcoh(X^2)}
\big(\Delta_*\mathcal{O}_X,\ \Delta_*\mathcal{O}_X \big)^{S^1}\otimes_{k[[\beta]]} k((\beta))
\end{align*}
Here, $\simeq^0$ follows from the definition of Hochschild cohomology, Definition \ref{defn:cohomo}, and Theorem \ref{thm:repn of ev and id}. $\simeq^1$ follows from the fact that Grothendieck duality $\D$ is an equivalence \eqref{eqn:gdcoh} and $\D(\cO_X) = \omega_X$.
$\simeq^2$ follows from Lemma \ref{lem:hocolim}.  Continuing from above, we have the following equivalences: All the tensor products are over $k[[\beta]]$.
\begin{align*}
RHom^{k}_{\Qcoh(X^2)}
\big(\Delta_*\mathcal{O}_X,\ \Delta_*\mathcal{O}_X \big)^{S^1}\otimes k((\beta))&\simeq^3 RHom^{k}_{\Qcoh(X)}
\big(\Delta^*\Delta_*\mathcal{O}_X,\ \mathcal{O}_X \big)^{S^1} \otimes k((\beta))\\ 
&\simeq^4 RHom^{k}_{\Qcoh(X)}
\big(p_*p^*\mathcal{O}_X,\ \mathcal{O}_X \big)^{S^1} \otimes k((\beta))\\
&\simeq^5 RHom^{k}_{\Qcoh(X)}
\big(p_*\mathcal{O}_{\mathcal{L}X},\ \mathcal{O}_X \big)^{S^1} \otimes k((\beta))\\
&\simeq^6 RHom^{k}_{\Qcoh(\mathcal{L} X)}
\big(\mathcal{O}_{\mathcal{L}X}, p^!\mathcal{O}_X \big)^{S^1} \otimes k((\beta))
\end{align*}
Here, $\simeq^3$ follows from adjunction $(\Delta^*,\Delta_*)$ in $\Qcoh$.
$\simeq^4$ follows from base change square for $\mathcal{L}X$ (see \eqref{eqn:loop stack}).
$\simeq^5$  follows from since $p^*\mathcal{O}_X \simeq \mathcal{O}_{\mathcal{L}X}$.
$\simeq^6$  follows from adjunction $(p_*,p^!)$ between $\Qcoh(\mathcal{L}X)$ and $\Qcoh(X)$. Continuing from above, we have the following equivalences: All the tensor products are over $k[[\beta]]$.
\begin{align*}
RHom^{k}_{\Qcoh(\mathcal{L} X)}
\big(\mathcal{O}_{\mathcal{L}X}, p^!\mathcal{O}_X \big)^{S^1} \otimes  k((\beta))&\simeq^7 R\Gamma(\mathcal{L}X, p^!\mathcal{O}_X)^{S^1} \otimes k((\beta)) \\
&\simeq^{8} R\Gamma\big(T[-1]I^{DM}X, 
q^!\mathcal{O}_X\big)^{B{\mathbb{G}_a}}\otimes k((\beta)) \\
&\simeq^{9} R\Gamma(I^{DM}X, \Sym(\mathbb{L}_{I^{DM}X}[1]) \otimes \pi^*\omega_X^{\vee}[-\dim X])^{B{\mathbb{G}_a}}\otimes k((\beta))\\
&\simeq^{10} R\Gamma(I^{DM}X, \Sym(\mathbb{L}^{\vee}_{I^{DM}X}[-1]) \otimes \det N [-c])^{B{\mathbb{G}_a}}\otimes k((\beta))
\end{align*}
Here, $\simeq^{7}$ follows from the identification of internal Hom with derived global sections.
$\simeq^{8}$ follows from identification of $\mathcal{L}X$ with $T[-1]I^{DM}X$ and translation of $S^1$-action to $B\mathbb{G}_a$-action by HKR theorem.
$\simeq^{9}$ follows from the computation of functions on $T[-1]I^{DM}X$ via symmetric algebra of the shifted cotangent complex with twist by $\omega_X$.
$\simeq^{10}$ follows from rewriting using $\mathbb{L}^\vee_{I^{DM}X}$ and expressing the twist via $\det N[-c]$ \cite[Corollary 4.15]{HKR2026}.
\noindent This completes the proof.
\end{proof}

Computing the $B\G_a$ fixed points, we have the following:

\begin{cor}\label{cor:cohomo}
    Let $(X,f)$ be a LG model satisfying the conditions of Theorem \ref{thm:Hochs cohomo main}. Then $HH^{\bullet}(MF^{\mathrm{coh}}(X,f)) $ is equivalent to the following 2-periodic complex over $k$:
    \begin{equation}
    \begin{split}
         R\Gamma\left(\Tot\left( \cdots \to \bigoplus_{i~ {\rm even}}\bigwedge^i\L^{\vee}_{I^{DM}X} \xrightarrow{((-)\wedge df\mid_{I^{DM}X})^{\vee}} \bigoplus_{i~ {\rm odd}}\bigwedge^i\L^{\vee}_{I^{DM}X}\to \cdots  \right) \otimes_k\right. \\ \left.\Tot\left(  \cdots \to \bigoplus_{j~ even}N'^j \xrightarrow{d'}  \bigoplus_{j~ odd}N'^j \to \cdots \right)\right)
    \end{split}
    \end{equation}
    where $d'$ is the sum of differentials $d_{N'}$ of the complex $N'^{\bullet}:=\det N[-c]$ and the differential by which $\epsilon$ acts on $N^{'\bullet}$.
\end{cor}

\begin{proof}
    The proof follows from \thmref{lem:de Rham action} and \thmref{lem:dr twist} applied to the last step of the equivalences in Theorem \ref{thm:Hochs cohomo main}.
\end{proof}

\subsection{HKR-type theorem for coherent matrix factorization categories of quasi-smooth derived schemes}
Let $(X,f)$ be a LG model where $X$ is a quasi-smooth derived scheme. Then we have $I^{DM}X = X$ by \defref{defn:orbifold inertia} and hence $N=\hocofib(T_{I^{DM}X} \to \pi^*T_X) = 0$. Therefore, as a particular case of Corollary \ref{cor:homo} and \ref{cor:cohomo} we get the following,
\begin{cor}\label{cor:q sm HH}
    We have the following isomorphisms of 2-periodic complexes,
    \begin{equation}
         HH_{\bullet}(MF^{\mathrm{coh}}(X,f)) \xrightarrow{\simeq} R\Gamma\left(\Tot\left( \cdots \to \bigoplus_{i~ {\rm even}}\bigwedge^i\L_{X} \xrightarrow{(-)\wedge df} \bigoplus_{i~ {\rm odd}}\bigwedge^i\L_{X} \to \cdots  \right)\right),
    \end{equation}
    and
     \begin{equation}
         HH^{\bullet}(MF^{\mathrm{coh}}(X,f)) \xrightarrow{\simeq} R\Gamma \left(\Tot\left( \cdots \to \bigoplus_{i~ {\rm even}}\bigwedge^i\L^{\vee}_{X} \xrightarrow{((-)\wedge df\mid_{X})^{\vee}} \bigoplus_{i~ {\rm odd}}\bigwedge^i\L^{\vee}_{X} \to \cdots \right)\right). 
    \end{equation}
\end{cor}

\subsection{Comparison of Preygel's result with other approaches in literature}\label{sec:comparison}
Various authors have proved several HKR-type theorems in the context of LG models. In this section, we discuss the correspondence between the  HKR type isomorphisms for matrix factorization categories of smooth $\DM$ stacks, a result due to Preygel \cite{Preygel:2011} and other approaches in literature. 

 Let $X$ be a smooth affine variety and $G$ be a finite group acting on $X$. 
In \cite[Section 2]{PV-DUKE}, Polishchuk and Vaintrob proved HKR-type isomorphisms for matrix factorization categories corresponding to the LG model $([X/G],f)$. In the same direction, there are works of C\u{a}ld\u{a}raru and Tu in \cite[Section 4]{caldararu-tu} and Segal in \cite{segal-closed-state} for curved $A_{\infty}$-algebras. 
In \cite[Corollary 4.6]{Halpern_Leistner_2020---arxiv--v2} Halpern-Leistner and Pomerleano proved an HKR type of isomorphism for $MF(X,f)$ where $X$ is a smooth $\DM$ stack and $f$ is a flat morphism following ideas developed by Preygel \cite{Preygel:2011}. 
In \cite{ballard-favero-katzarkov-kernel} Ballard, Favero, and Katzarkov proved an equivariant analogue of this.
In \cite[Section 6]{chain-level-hkr-type-map-chern--kim} Chung, Kim and Kim proved chain level HKR type isomorphism for $MF(X,f)$ where $X$ is a smooth separated scheme or global quotient stack obtained by action of a finite group on a smooth separated scheme which generalizes the HKR type isomorphism for smooth affine schemes by Brown and Walker in \cite[Theorem 5.19]{brown-walker-NG}.
A chain level version of this theorem has also been studied by Choa, Kim, and the second author in \cite[Theorem 4.19]{CKS}.

\subsubsection{HKR for matrix factorization categories on smooth schemes and stacks}
 Let $(X,f)$ be a LG model where $X$ is a smooth Deligne-Mumford stack and $f: X\to \A^1$ is a flat morphism.
As $X$ is a smooth $\DM$ stack over a field of characteristic $0$,  $D^bCoh(X) \simeq Perf(X)$. It follows from the \defref{defn:sing cat} that the categories $$Sing^{\infty}(X) = Sing(X) = 0.$$ 
Hence, $MF^{coh, \infty}(X) = Sing^{\infty}(X_0)$ and $MF^{coh}(X) = Sing(X_0)$ where $X_0$ is the zero locus of $f$. Note that as $f$ is flat, the derived pullback in \eqref{eqn:zero-locus} coincides with the classical pullback.
Again, since $X$ is smooth over $k$, $IX$ is also smooth over $k$ and thus  $\L_{IX/k} \simeq \Omega_{IX/k}$, the sheaf of K\"ahler differential of $IX$. So $\bigwedge^i \L_{IX/k} = \bigwedge^i \Omega^1_{IX/k} =: \Omega^i_{IX/k}.$

Let $(X,f)$ be an LG model; where $X=\Spec A$ is a smooth affine scheme over $k$. Then $\mathcal{L}X \simeq \Spec (A \otimes^L_{A\otimes^L A} A)$ and $T[-1]I^{DM}X = \Spec \Sym(\Omega^{\bullet}_{A/k}[1]).$ Thus, by \cite[Proposition 5.6]{HKR2026} we have the following commutative diagram of commutative $A$-algebras:
\begin{equation}\label{eqn:affinization for affine}
    \begin{tikzcd}[cramped]
	{\Sym(\Omega^{\bullet}_{A/k}[1])} && {A\otimes^L_{A\otimes^L A}A} \\
	& A
	\arrow["{R\Gamma(\wh{aff}^*)}"', from=1-3, to=1-1]
	\arrow["{\alpha'}", from=2-2, to=1-1]
	\arrow["\alpha"', from=2-2, to=1-3]
\end{tikzcd}
\end{equation}
where $\alpha$ and $\alpha'$ are structure morphisms defined by $\alpha'(a) = d_{dR}(a) \in \Omega^1_{A/k} \hookrightarrow \Sym(\Omega^{\bullet}_{A/k}[1])$ and $\alpha(a) = a \in A \hookrightarrow A\otimes^L_{A\otimes^L A}A $ respectively. Therefore, 
\begin{equation}\label{eqn:image of r gamma}
    R\Gamma(\wh{aff}^*)(a) = d_{dR}(a).
\end{equation}
Since the horizontal equivalence $R\Gamma(\wh{aff}^*)$ in \eqref{eqn:affinization for affine} is equivariant with respect to the  $S^1 \times X$ action over $X$, it preserves the algebra structure on the domain and co-domain \cite[Proposition 5.6]{HKR2026}. The graded commutative algebra structure on $HH_{\bullet} = A \otimes^L_{A\otimes^L A} A$ is given by \emph{shuffle product} \cite[Cor. 4.2.7]{loday_cyclic_homology}, and the graded commutative algebra structure on target of $R\Gamma(\wh{aff}^*)$ is given by symmetric algebra structure. Since $R\Gamma(\wh{aff}^*)$ is an isomorphism by \cite[Section 4.2.9]{loday_cyclic_homology} we get $(R\Gamma(\wh{aff}^*))^{-1}$  is the same as the anti-symmetrization map (see \cite[Section 1.3.4]{loday_cyclic_homology}).
Now, the $S^1$ action on $HH_{\bullet}(A)$ induced from the potential is given by the connes operator (a degree $-1$ differential denoted by $B_w$ in \cite[Section 4]{platt}, see also \cite{platt2012cherncharacterglobalmatrix}) on \emph{curved complete bar complex} denoted by $\wh{\mathfrak{B}ar}_{\wt f} \in MF(X\times X, f \boxplus -f)$ \cite[Section 4]{platt}, \cite[Section 1]{yekutieli} where $\wt{f}:= f\boxplus -f$. By \cite[Lemma 4.2]{platt}, we have $\wh{\mathfrak{B}ar}_f \simeq \Delta_*(\cO_X)$ where $\cO_X$ is seen as a factorization of $0:X\to \A^1$ as in Section \ref{sec:o_x as facto}.
Now, $\Delta^*\wh{\mathfrak{B}ar}_f$ is given by totalization of the double complex in \cite[Equation (2), pg. 14]{platt}.
By multiplicativity of $R\Gamma(\wh{aff}^*)$, $B_f$ corresponds to the degree $-1$ differential $(-)\wedge d_{dR}f$ on $\Sym(\Omega^{\bullet}_{A/k}[1])$. 
Hence, by multiplicativity of $R\Gamma(\wh{aff}^*)$ and \eqref{eqn:image of r gamma} the isomorphism in \eqref{cor:homo} reduces to the following form:
\[R\Gamma \Delta^*\wh{\mathfrak{B}ar}_f \xrightarrow{\simeq} R\Gamma \left( \cdots \to \bigoplus_{i~ {\rm even}}\Omega^i_{X} \xrightarrow{(-)\wedge df} \bigoplus_{i~ {\rm odd}}\Omega^i_{X} \to\cdots  \right)\]
\begin{equation}\label{eqn:hkr local form}
    a_0[a_1|a_2|\cdots|a_n] \mapsto \frac{1}{n!} a_0 d_{dR}f \wedge d_{dR}(a_1)\wedge d_{dR}(a_2)\wedge \cdots \wedge d_{dR}(a_n)
\end{equation}
which is  as in \cite[pg. 326]{caldararu-tu}, \cite[Proposition 8.2.4]{Preygel:2011} where $a_0[a_1|a_2|\cdots|a_n]$ denotes an $n$-chain of the bar resolution.
So by the above discussion, we get the algebra homomorphism $R\Gamma(\wh{aff}^*)$ takes the form \eqref{eqn:hkr local form} on the $n$-chains of the Hochschild chain complex.

Now, let $G$ be a finite group acting on a smooth affine scheme $X=\Spec A$. Consider the quotient stack $\sX = [X/G]$. Let $f:\sX \to \A^1$ be a flat morphism. 
Then the inertia stack has the following decomposition \cite[Proposition 6.5]{HKR2026}, \cite[Section 3]{vistoli-gw-th}:
\begin{equation}
    I^{DM}\sX \simeq \bigsqcup_{r=1}^{\infty}I_{\mu_r}\sX
\end{equation}
where $I_{\mu_r}\sX$ be the $r$-th cyclotomic inertia of $\sX$.
Consider the following commutative diagram with the following notations:
$\underline{\ol{MC}}^{II}(-)$ denotes the mixed Hochschild complex of second kind, $\mathfrak{tw}_{\chi}$ denotes the twisting automorphism corresponding to $\chi \in \wh\mu_r$, $nat$ denotes the natural map. For details please see \cite[Section 2.1 \& 4.7]{CKS}.  Note for ${(aff)}^*$ the twisting action corresponds to the action described in the twisted sector by the local system with monodromy given by the character $\chi$ (see \cite[Proposition 4.10]{HKR2026}).
 The unique upward vertical map is given by quasi-Morita invariance \cite[Section 2.2.2]{CKS}. 
The arrow with ``$HKR$'' label is the HKR type isomorphism of \cite[Section 3.3]{segal-closed-state} and \cite[Section 6]{caldararu-tu}.

\[\begin{tikzcd}[cramped]
	{\ol{MC}(MF(X,f))} & {R\Gamma\underline{\ol{MC}}^{II}(qMF(\sX,f))} & {\oplus_{r, \chi}\underline{\ol{MC}}^{II}(qMF^{\chi}(I_{\mu_r}\sX,f\mid_{I_{\mu_r}\sX}))} \\
	&& {\oplus_{r, \chi}\underline{\ol{MC}}^{II}(qMF^{\chi}(I_{\mu_r}\sX,f\mid_{I_{\mu_r}\sX}))} \\
	&& {\oplus_rR\Gamma(I_{\mu_r}\sX, \underline{\ol{MC}}^{II}(qMF_{I_{\mu_r}\sX,f\mid_{I_{\mu_r}\sX}}(-)))} \\
	{R\Gamma(\Omega^{\bullet}_{I\sX},df\mid_{I\sX},ud)} && {R\Gamma(\underline{\ol{MC}}^{II}(\cO_{I\sX},f\mid_{I\sX}))}
	\arrow["\simeq", from=1-1, to=1-2]
	\arrow["{R\Gamma(\wh{(aff)}^*)}"', from=1-1, to=4-1]
	\arrow["pullback", from=1-2, to=1-3]
	\arrow["{\oplus \mathfrak{tw}_{\chi}}", from=1-3, to=2-3]
	\arrow["nat", from=2-3, to=3-3]
	\arrow["\simeq"', from=4-3, to=3-3]
	\arrow["{HKR}", from=4-3, to=4-1]
\end{tikzcd}\]
By the discussion above and \cite[Theorem 4.19]{CKS}, \cite[Section 6]{caldararu-tu} and \cite[Section 3.3]{segal-closed-state}, it follows that the above diagram commutes. 

We summarize the discussion in the following corollary: 
\begin{cor}
Let $(X,f)$ be an LG model where we assume that  $X$ is a smooth $\DM$-stack and $f:X\to \A^1$ is flat. Then,
    \begin{equation}\label{eqn:smooth stack hh}
    HH_{\bullet}(MF(X,f)) \xrightarrow{\simeq} R\Gamma\left( \cdots \to \bigoplus_{i~ {\rm even}}\Omega^i_{IX} \xrightarrow{(-)\wedge df} \bigoplus_{i~ {\rm odd}}\Omega^i_{IX} \to\cdots  \right),
\end{equation}
and
     \begin{equation}
         HH^{\bullet}(MF(X,f)) \xrightarrow{\simeq} R\Gamma\left( \cdots \to \bigoplus_{i~ {\rm even}}\Omega^{i\, \vee}_{IX} \xrightarrow{((-)\wedge df\mid_{IX})^{\vee}} \bigoplus_{i~ {\rm odd}}\Omega^{i\, \vee}_{IX}\to \cdots  \right) \otimes R\Gamma\left(k[\beta] \wh\otimes \det N[-c]\right) 
    \end{equation}
    and the isomorphism \eqref{eqn:smooth stack hh} coincides with the isomorphism of \cite[Theorem 4.19]{CKS} under the identification of the singularity category with the matrix factorization category by  Orlov's equivalence.
\end{cor}

\printbibliography

@article{swan1996hochschild,
  title={Hochschild cohomology of quasiprojective schemes},
  author={Swan, Richard G},
  journal={Journal of Pure and Applied Algebra},
  volume={110},
  number={1},
  pages={57--80},
  year={1996},
  publisher={Elsevier}
}

@article{Ben_Zvi_2012,
	author = {Ben-Zvi, David and Nadler, David},
	doi = {10.1112/jtopol/jts007},
	issn = {1753-8416},
	journal = {Journal of Topology},
	month = mar,
	number = {2},
	pages = {377--430},
	publisher = {Wiley},
	title = {Loop spaces and connections},
	url = {http://dx.doi.org/10.1112/jtopol/jts007},
	volume = {5},
	year = {2012}}

@article{Preygel:2011,
	author = {Anatoly Preygel},
	eprint = {1101.5834},
	month = {01},
	title = {Thom-Sebastiani \& Duality for Matrix Factorizations},
	url = {https://arxiv.org/pdf/1101.5834.pdf},
	year = {2011}}

@book {HTT,
    shorthand = {HTT},
    AUTHOR = {Lurie, Jacob},
     TITLE = {Higher topos theory},
    SERIES = {Annals of Mathematics Studies},
    VOLUME = {170},
 PUBLISHER = {Princeton University Press, Princeton, NJ},
      YEAR = {2009},
     PAGES = {xviii+925},
      ISBN = {978-0-691-14049-0},
   MRCLASS = {18-02 (18B25 18E35 18G30 18G55 55U40)},
  MRNUMBER = {2522659},
MRREVIEWER = {Mark\ Hovey},
       DOI = {10.1515/9781400830558},
       URL = {https://doi.org/10.1515/9781400830558}
}

@article {MR3877165,
    AUTHOR = {Blanc, Anthony and Robalo, Marco and To\"en, Bertrand and
              Vezzosi, Gabriele},
     TITLE = {Motivic realizations of singularity categories and vanishing
              cycles},
   JOURNAL = {J. \'Ec. polytech. Math.},
  FJOURNAL = {Journal de l'\'Ecole polytechnique. Math\'ematiques},
    VOLUME = {5},
      YEAR = {2018},
     PAGES = {651--747},
      ISSN = {2429-7100,2270-518X},
   MRCLASS = {14F42 (14A22 14F05 16S38 18D10 18E30 19E08 32S30)},
  MRNUMBER = {3877165},
MRREVIEWER = {Jian\ Min\ Chen},
       DOI = {10.5802/jep.81},
       URL = {https://doi.org/10.5802/jep.81},
}

@article {homotopy-th-of-dg-cats-Tabuada,
    AUTHOR = {Tabuada, Gon\c calo},
     TITLE = {Homotopy theory of dg categories via localizing pairs and
              {D}rinfeld's dg quotient},
   JOURNAL = {Homology Homotopy Appl.},
  FJOURNAL = {Homology, Homotopy and Applications},
    VOLUME = {12},
      YEAR = {2010},
    NUMBER = {1},
     PAGES = {187--219},
      ISSN = {1532-0073,1532-0081},
   MRCLASS = {18D20 (18G55)},
  MRNUMBER = {2607415},
MRREVIEWER = {Philippe\ Gaucher},
       DOI = {10.4310/hha.2010.v12.n1.a11},
       URL = {https://doi.org/10.4310/hha.2010.v12.n1.a11},
}

@article {chen,
    AUTHOR = {Chen, Harrison},
     TITLE = {Equivariant localization and completion in cyclic homology and
              derived loop spaces},
   JOURNAL = {Adv. Math.},
  FJOURNAL = {Advances in Mathematics},
    VOLUME = {364},
      YEAR = {2020},
     PAGES = {107005, 56},
      ISSN = {0001-8708,1090-2082},
   MRCLASS = {14F08 (14A20 55P35)},
  MRNUMBER = {4060042},
MRREVIEWER = {Sadok\ Kallel},
       DOI = {10.1016/j.aim.2020.107005},
       URL = {https://doi.org/10.1016/j.aim.2020.107005},
}

@misc{HA,
    shorthand = {HA},
  author       = {Jacob Lurie},
  title        = {Higher algebra},
  howpublished = {\url{http://www.math.harvard.edu/~lurie/}},
  year         = {2017},
}

@article {pippi,
    AUTHOR = {Pippi, Massimo},
     TITLE = {On some (co)homological invariants of coherent matrix
              factorizations},
   JOURNAL = {J. Noncommut. Geom.},
  FJOURNAL = {Journal of Noncommutative Geometry},
    VOLUME = {17},
      YEAR = {2023},
    NUMBER = {4},
     PAGES = {1299--1334},
      ISSN = {1661-6952,1661-6960},
   MRCLASS = {14F08 (32S30)},
  MRNUMBER = {4653786},
MRREVIEWER = {Amin\ Gholampour},
       DOI = {10.4171/jncg/515},
       URL = {https://doi.org/10.4171/jncg/515},
}

@misc{sag,
  author       = {Jacob Lurie},
  title        = {Spectral Algebraic Geometry},
  howpublished = {\url{https://www.math.ias.edu/~lurie/papers/SAG-rootfile.pdf}},
  year         = {2017},
  label = {SAG},
}

@misc{HKR2026,
      title={Hochschild-Kostant-Rosenberg isomorphism for derived Deligne-Mumford stacks}, 
      author={Lie Fu and Mauro Porta and Sarah Scherotzke and Nicolò Sibilla},
      year={2026},
      eprint={2509.00501},
      archivePrefix={arXiv},
      primaryClass={math.AG},
      note={\url{https://arxiv.org/abs/2509.00501}}, 
}

@article {Toen_Atiyah-class,
    AUTHOR = {Sch\"urg, Timo and To\"en, Bertrand and Vezzosi, Gabriele},
     TITLE = {Derived algebraic geometry, determinants of perfect complexes,
              and applications to obstruction theories for maps and
              complexes},
   JOURNAL = {J. Reine Angew. Math.},
  FJOURNAL = {Journal f\"ur die Reine und Angewandte Mathematik. [Crelle's
              Journal]},
    VOLUME = {702},
      YEAR = {2015},
     PAGES = {1--40},
      ISSN = {0075-4102,1435-5345},
   MRCLASS = {14D20 (14F05)},
  MRNUMBER = {3341464},
MRREVIEWER = {Nicolas\ Perrin},
       DOI = {10.1515/crelle-2013-0037},
       URL = {https://doi.org/10.1515/crelle-2013-0037},
}

@article {HAG_2,
    AUTHOR = {To\"en, Bertrand and Vezzosi, Gabriele},
     TITLE = {Homotopical algebraic geometry. {II}. {G}eometric stacks and
              applications},
   JOURNAL = {Mem. Amer. Math. Soc.},
  FJOURNAL = {Memoirs of the American Mathematical Society},
    VOLUME = {193},
      YEAR = {2008},
    NUMBER = {902},
     PAGES = {x+224},
      ISSN = {0065-9266,1947-6221},
   MRCLASS = {14A20 (18F10 18F20 18G55 55P42 55U40)},
  MRNUMBER = {2394633},
MRREVIEWER = {Paul\ Arne\ \O stv\ae r},
       DOI = {10.1090/memo/0902},
       URL = {https://doi.org/10.1090/memo/0902},
}

@article {Shifted_cotangent,
    AUTHOR = {Calaque, Damien},
     TITLE = {Shifted cotangent stacks are shifted symplectic},
   JOURNAL = {Ann. Fac. Sci. Toulouse Math. (6)},
  FJOURNAL = {Annales de la Facult\'e{} des Sciences de Toulouse.
              Math\'ematiques. S\'erie 6},
    VOLUME = {28},
      YEAR = {2019},
    NUMBER = {1},
     PAGES = {67--90},
      ISSN = {0240-2963,2258-7519},
   MRCLASS = {14A20 (14F05)},
  MRNUMBER = {3940792},
MRREVIEWER = {Feng\ Qu},
       DOI = {10.5802/afst.1593},
       URL = {https://doi.org/10.5802/afst.1593},
}

@article {pippi_HS,
    AUTHOR = {Pippi, Massimo},
     TITLE = {On the structure of dg categories of relative singularities},
   JOURNAL = {High. Struct.},
  FJOURNAL = {Higher Structures},
    VOLUME = {6},
      YEAR = {2022},
    NUMBER = {1},
     PAGES = {375--402},
      ISSN = {2209-0606},
   MRCLASS = {14F08 (14A22 14B05 18G80)},
  MRNUMBER = {4456599},
MRREVIEWER = {Zoran\ \v Skoda},
}

@article {HLP_hodge,
    AUTHOR = {Halpern-Leistner, Daniel and Pomerleano, Daniel},
     TITLE = {Equivariant {H}odge theory and noncommutative geometry},
   JOURNAL = {Geom. Topol.},
  FJOURNAL = {Geometry \& Topology},
    VOLUME = {24},
      YEAR = {2020},
    NUMBER = {5},
     PAGES = {2361--2433},
      ISSN = {1465-3060,1364-0380},
   MRCLASS = {14A22 (14C30 14F08 19D55 19L47)},
  MRNUMBER = {4194295},
MRREVIEWER = {Pieter\ Belmans},
       DOI = {10.2140/gt.2020.24.2361},
       URL = {https://doi.org/10.2140/gt.2020.24.2361},
}

@book {loday_cyclic_homology,
    AUTHOR = {Loday, Jean-Louis},
     TITLE = {Cyclic homology},
    SERIES = {Grundlehren der mathematischen Wissenschaften [Fundamental
              Principles of Mathematical Sciences]},
    VOLUME = {301},
      NOTE = {Appendix E by Mar\'ia O. Ronco},
 PUBLISHER = {Springer-Verlag, Berlin},
      YEAR = {1992},
     PAGES = {xviii+454},
      ISBN = {3-540-53339-7},
   MRCLASS = {19D55 (17B56 18E25 55N91)},
  MRNUMBER = {1217970},
MRREVIEWER = {Jerry\ Lodder},
       DOI = {10.1007/978-3-662-21739-9},
       URL = {https://doi.org/10.1007/978-3-662-21739-9},
}

@article {TV_multi_hkr,
    AUTHOR = {To\"en, Bertrand and Vezzosi, Gabriele},
     TITLE = {Alg\`ebres simpliciales {$S^1$}-\'equivariantes, th\'eorie de
              de {R}ham et th\'eor\`emes {HKR} multiplicatifs},
   JOURNAL = {Compos. Math.},
  FJOURNAL = {Compositio Mathematica},
    VOLUME = {147},
      YEAR = {2011},
    NUMBER = {6},
     PAGES = {1979--2000},
      ISSN = {0010-437X,1570-5846},
   MRCLASS = {18G55 (14F40 16E45 18G30 55U10)},
  MRNUMBER = {2862069},
MRREVIEWER = {Kenneth\ A.\ Brown},
       DOI = {10.1112/S0010437X11005501},
       URL = {https://doi.org/10.1112/S0010437X11005501},
}

@article {gaitsgory,
    AUTHOR = {Drinfeld, Vladimir and Gaitsgory, Dennis},
     TITLE = {On some finiteness questions for algebraic stacks},
   JOURNAL = {Geom. Funct. Anal.},
  FJOURNAL = {Geometric and Functional Analysis},
    VOLUME = {23},
      YEAR = {2013},
    NUMBER = {1},
     PAGES = {149--294},
      ISSN = {1016-443X,1420-8970},
   MRCLASS = {14A20 (14F05 14F10 18Dxx)},
  MRNUMBER = {3037900},
MRREVIEWER = {Pawel\ Sosna},
       DOI = {10.1007/s00039-012-0204-5},
       URL = {https://doi.org/10.1007/s00039-012-0204-5},
}

@article {Orlov-mf,
    AUTHOR = {Orlov, Dmitri},
     TITLE = {Matrix factorizations for nonaffine {LG}-models},
   JOURNAL = {Math. Ann.},
  FJOURNAL = {Mathematische Annalen},
    VOLUME = {353},
      YEAR = {2012},
    NUMBER = {1},
     PAGES = {95--108},
      ISSN = {0025-5831,1432-1807},
   MRCLASS = {14F05 (18E30)},
  MRNUMBER = {2910782},
MRREVIEWER = {Pawel\ Sosna},
       DOI = {10.1007/s00208-011-0676-x},
       URL = {https://doi.org/10.1007/s00208-011-0676-x},
}

@article {CKS,
    AUTHOR = {Choa, Dongwook and Kim, Bumsig and Sreedhar, Bhamidi},
     TITLE = {Riemann-{R}och for stacky matrix factorizations},
   JOURNAL = {Forum Math. Sigma},
  FJOURNAL = {Forum of Mathematics. Sigma},
    VOLUME = {10},
      YEAR = {2022},
     PAGES = {Paper No. e108, 29},
      ISSN = {2050-5094},
   MRCLASS = {14A22 (14A20 16E40 18G80)},
  MRNUMBER = {4519062},
MRREVIEWER = {Ji-Wei\ He},
       DOI = {10.1017/fms.2022.99},
       URL = {https://doi.org/10.1017/fms.2022.99},
}

@article {champes-affine,
    AUTHOR = {To\"en, Bertrand},
     TITLE = {Champs affines},
   JOURNAL = {Selecta Math. (N.S.)},
  FJOURNAL = {Selecta Mathematica. New Series},
    VOLUME = {12},
      YEAR = {2006},
    NUMBER = {1},
     PAGES = {39--135},
      ISSN = {1022-1824,1420-9020},
   MRCLASS = {14F35 (14A20 18F10 55U35)},
  MRNUMBER = {2244263},
MRREVIEWER = {Mark\ Hovey},
       DOI = {10.1007/s00029-006-0019-z},
       URL = {https://doi.org/10.1007/s00029-006-0019-z},
}

@book {Gaitsgory-book-1,
    AUTHOR = {Gaitsgory, Dennis and Rozenblyum, Nick},
     TITLE = {A study in derived algebraic geometry. {V}ol. {I}.
              {C}orrespondences and duality},
    SERIES = {Mathematical Surveys and Monographs},
    VOLUME = {221},
 PUBLISHER = {American Mathematical Society, Providence, RI},
      YEAR = {2017},
     PAGES = {xl+533},
      ISBN = {978-1-4704-3569-1},
   MRCLASS = {14F05 (18D05 18G55)},
  MRNUMBER = {3701352},
MRREVIEWER = {Adrian\ Langer},
       DOI = {10.1090/surv/221.1},
       URL = {https://doi.org/10.1090/surv/221.1},
}

@incollection {khan_derived-geom,
    AUTHOR = {Khan, Adeel A.},
     TITLE = {An introduction to derived algebraic geometry},
 BOOKTITLE = {Moduli spaces, virtual invariants and shifted symplectic
              structures},
    SERIES = {KIAS Springer Ser. Math.},
    VOLUME = {4},
     PAGES = {1--35},
 PUBLISHER = {Springer, Singapore},
      YEAR = {2025},
      ISBN = {978-981-97-8248-2},
   MRCLASS = {14A30 (14D23 14F08)},
  MRNUMBER = {4898354},
       DOI = {10.1007/978-981-97-8249-9\_1},
       URL = {https://doi.org/10.1007/978-981-97-8249-9_1},
}

@article {vit_khan,
    AUTHOR = {Khan, Adeel A.},
     TITLE = {Virtual excess intersection theory},
   JOURNAL = {Ann. K-Theory},
  FJOURNAL = {Annals of K-Theory},
    VOLUME = {6},
      YEAR = {2021},
    NUMBER = {3},
     PAGES = {559--570},
      ISSN = {2379-1683,2379-1691},
   MRCLASS = {14C35 (14A20 14A30 14C17 19Exx)},
  MRNUMBER = {4310329},
MRREVIEWER = {Gabriel\ Angelini-Knoll},
       DOI = {10.2140/akt.2021.6.559},
       URL = {https://doi.org/10.2140/akt.2021.6.559},
}

@article {khan-ravi-cohomo-alg-stacks,
    AUTHOR = {Khan, Adeel A. and Ravi, Charanya},
     TITLE = {Generalized cohomology theories for algebraic stacks},
   JOURNAL = {Adv. Math.},
  FJOURNAL = {Advances in Mathematics},
    VOLUME = {458},
      YEAR = {2024},
     PAGES = {Paper No. 109975, 104},
      ISSN = {0001-8708,1090-2082},
   MRCLASS = {14A20 (14C15 14F42 19E08 55N20)},
  MRNUMBER = {4811546},
MRREVIEWER = {Bj\o rn\ Ian\ Dundas},
       DOI = {10.1016/j.aim.2024.109975},
       URL = {https://doi.org/10.1016/j.aim.2024.109975},
}

@article {khan-ravi-prep,
    AUTHOR = {Khan, Adeel A. and Ravi, Charanya},
     TITLE = {In preparation},
  
      YEAR = {2026},
     
}

@article {derived-morita-theory,
    AUTHOR = {To\"en, Bertrand},
     TITLE = {The homotopy theory of {$dg$}-categories and derived {M}orita
              theory},
   JOURNAL = {Invent. Math.},
  FJOURNAL = {Inventiones Mathematicae},
    VOLUME = {167},
      YEAR = {2007},
    NUMBER = {3},
     PAGES = {615--667},
      ISSN = {0020-9910,1432-1297},
   MRCLASS = {18D05 (18E30 18G55 19D55)},
  MRNUMBER = {2276263},
MRREVIEWER = {Mark\ Hovey},
       DOI = {10.1007/s00222-006-0025-y},
       URL = {https://doi.org/10.1007/s00222-006-0025-y},
}

@article {kth-gth-khan,
    AUTHOR = {Khan, Adeel A.},
     TITLE = {K-theory and {G}-theory of derived algebraic stacks},
   JOURNAL = {Jpn. J. Math.},
  FJOURNAL = {Japanese Journal of Mathematics},
    VOLUME = {17},
      YEAR = {2022},
    NUMBER = {1},
     PAGES = {1--61},
      ISSN = {0289-2316,1861-3624},
   MRCLASS = {19E08 (14A20 14A30)},
  MRNUMBER = {4397935},
       DOI = {10.1007/s11537-021-2110-9},
       URL = {https://doi.org/10.1007/s11537-021-2110-9},
}

@misc{hlp-theta-strata,
      title={Derived $\Theta$-stratifications and the $D$-equivalence conjecture}, 
      author={Daniel Halpern-Leistner},
      year={2021},
      eprint={2010.01127},
      archivePrefix={arXiv},
      primaryClass={math.AG},
      note={\url{https://arxiv.org/abs/2010.01127}}, 
}

@article {singular-support,
    AUTHOR = {Arinkin, D. and Gaitsgory, D.},
     TITLE = {Singular support of coherent sheaves and the geometric
              {L}anglands conjecture},
   JOURNAL = {Selecta Math. (N.S.)},
  FJOURNAL = {Selecta Mathematica. New Series},
    VOLUME = {21},
      YEAR = {2015},
    NUMBER = {1},
     PAGES = {1--199},
      ISSN = {1022-1824,1420-9020},
   MRCLASS = {14D24 (14A20 14F05 22E57)},
  MRNUMBER = {3300415},
MRREVIEWER = {Richard\ P.\ Thomas},
       DOI = {10.1007/s00029-014-0167-5},
       URL = {https://doi.org/10.1007/s00029-014-0167-5},
}

@misc{indcoherentsheaves,
      title={Ind-coherent sheaves}, 
      author={Dennis Gaitsgory},
      year={2012},
      eprint={1105.4857},
      archivePrefix={arXiv},
      primaryClass={math.AG},
      note={\url{https://arxiv.org/abs/1105.4857}}, 
}

@incollection {quillen,
    AUTHOR = {Quillen, Daniel},
     TITLE = {Higher algebraic {$K$}-theory: {I} [MR0338129]},
 BOOKTITLE = {Cohomology of groups and algebraic {$K$}-theory},
    SERIES = {Adv. Lect. Math. (ALM)},
    VOLUME = {12},
     PAGES = {413--478},
 PUBLISHER = {Int. Press, Somerville, MA},
      YEAR = {2010},
      ISBN = {978-1-57146-144-5},
   MRCLASS = {19Dxx},
  MRNUMBER = {2655184},
}

@misc{devissage-algebraic-k-theory-small,
      title={D\'evissage for Algebraic K-theory of Small Stable $\infty$-categories}, 
      author={Chunhui Wei},
      year={2026},
      eprint={2601.14626},
      archivePrefix={arXiv},
      primaryClass={math.KT},
      note ={\url{https://arxiv.org/abs/2601.14626}}, 
}

@article {det-heleodoro,
    AUTHOR = {Heleodoro, Aron},
     TITLE = {Determinant map for the prestack of {T}ate objects},
   JOURNAL = {Selecta Math. (N.S.)},
  FJOURNAL = {Selecta Mathematica. New Series},
    VOLUME = {26},
      YEAR = {2020},
    NUMBER = {5},
     PAGES = {Paper No. 76, 57},
      ISSN = {1022-1824,1420-9020},
   MRCLASS = {14A20 (14A30 18G99 22E67)},
  MRNUMBER = {4172986},
MRREVIEWER = {Federico\ Scavia},
       DOI = {10.1007/s00029-020-00604-3},
       URL = {https://doi.org/10.1007/s00029-020-00604-3},
}

@article {k-th-tabuada,
    AUTHOR = {Blumberg, Andrew J. and Gepner, David and Tabuada, Gon\c calo},
     TITLE = {A universal characterization of higher algebraic {$K$}-theory},
   JOURNAL = {Geom. Topol.},
  FJOURNAL = {Geometry \& Topology},
    VOLUME = {17},
      YEAR = {2013},
    NUMBER = {2},
     PAGES = {733--838},
      ISSN = {1465-3060,1364-0380},
   MRCLASS = {19D10 (18D20 19D25 19D55 55N15 55U40)},
  MRNUMBER = {3070515},
MRREVIEWER = {Ross\ Staffeldt},
       DOI = {10.2140/gt.2013.17.733},
       URL = {https://doi.org/10.2140/gt.2013.17.733},
}

@misc{raksit-hochschildhomologyderivedrham,
      title={Hochschild homology and the derived de Rham complex revisited}, 
      author={Arpon Raksit},
      year={2026},
      eprint={2007.02576},
      archivePrefix={arXiv},
      primaryClass={math.AG},
      note={\url{https://arxiv.org/abs/2007.02576}}, 
}

@article {bms,
    AUTHOR = {Bhatt, Bhargav and Morrow, Matthew and Scholze, Peter},
     TITLE = {Topological {H}ochschild homology and integral {$p$}-adic
              {H}odge theory},
   JOURNAL = {Publ. Math. Inst. Hautes \'Etudes Sci.},
  FJOURNAL = {Publications Math\'ematiques. Institut de Hautes \'Etudes
              Scientifiques},
    VOLUME = {129},
      YEAR = {2019},
     PAGES = {199--310},
      ISSN = {0073-8301,1618-1913},
   MRCLASS = {14F30 (13A35)},
  MRNUMBER = {3949030},
MRREVIEWER = {Lance\ Edward\ Miller},
       DOI = {10.1007/s10240-019-00106-9},
       URL = {https://doi.org/10.1007/s10240-019-00106-9},
}

@phdthesis{marangoni:tel-02957674,
  TITLE = {{On Derived de Rham cohomology}},
  AUTHOR = {Marangoni, Davide},
  URL = {https://theses.hal.science/tel-02957674},
  NUMBER = {2020BORD0095},
  SCHOOL = {{Universit{\'e} de Bordeaux ; Universit{\`a} degli studi (Milan, Italie)}},
  YEAR = {2020},
  MONTH = Jul,
  TYPE = {Theses},
  HAL_ID = {tel-02957674},
  HAL_VERSION = {v1},
}

@book {Thesis_Lurie,
    AUTHOR = {Lurie, Jacob},
     TITLE = {Derived algebraic geometry},
      NOTE = {Thesis (Ph.D.)--Massachusetts Institute of Technology},
 PUBLISHER = {ProQuest LLC, Ann Arbor, MI},
      YEAR = {2004},
     PAGES = {(no paging)},
   MRCLASS = {99-05},
  MRNUMBER = {2717174},
       URL =
              {http://gateway.proquest.com/openurl?url_ver=Z39.88-2004&rft_val_fmt=info:ofi/fmt:kev:mtx:dissertation&res_dat=xri:pqdiss&rft_dat=xri:pqdiss:0806251},
}

@article {thh-scholze,
    AUTHOR = {Nikolaus, Thomas and Scholze, Peter},
     TITLE = {On topological cyclic homology},
   JOURNAL = {Acta Math.},
  FJOURNAL = {Acta Mathematica},
    VOLUME = {221},
      YEAR = {2018},
    NUMBER = {2},
     PAGES = {203--409},
      ISSN = {0001-5962,1871-2509},
   MRCLASS = {55U35 (16E40 18E30 19D99)},
  MRNUMBER = {3904731},
MRREVIEWER = {Geoffrey\ M. L. Powell},
       DOI = {10.4310/ACTA.2018.v221.n2.a1},
       URL = {https://doi.org/10.4310/ACTA.2018.v221.n2.a1},
}

@misc{hoyois-homotopy-fixed-points-circle,
      title={The homotopy fixed points of the circle action on Hochschild homology}, 
      author={Marc Hoyois},
      year={2018},
      eprint={1506.07123},
      archivePrefix={arXiv},
      primaryClass={math.KT},
      note={\url{https://arxiv.org/abs/1506.07123}}, 
}

@article {iwanari-hochschild-cohomology,
    AUTHOR = {Iwanari, Isamu},
     TITLE = {Differential calculus of {H}ochschild pairs for
              infinity-categories},
   JOURNAL = {SIGMA Symmetry Integrability Geom. Methods Appl.},
  FJOURNAL = {SIGMA. Symmetry, Integrability and Geometry. Methods and
              Applications},
    VOLUME = {16},
      YEAR = {2020},
     PAGES = {Paper No. 097, 57},
      ISSN = {1815-0659},
   MRCLASS = {16E40 (18M60 18N60)},
  MRNUMBER = {4156865},
MRREVIEWER = {Beno\^it\ Fresse},
       DOI = {10.3842/SIGMA.2020.097},
       URL = {https://doi.org/10.3842/SIGMA.2020.097},
}

@article {noohi,
    AUTHOR = {Noohi, Behrang},
     TITLE = {Mapping stacks of topological stacks},
   JOURNAL = {J. Reine Angew. Math.},
  FJOURNAL = {Journal f\"ur die Reine und Angewandte Mathematik. [Crelle's
              Journal]},
    VOLUME = {646},
      YEAR = {2010},
     PAGES = {117--133},
      ISSN = {0075-4102,1435-5345},
   MRCLASS = {57R19 (55P35 55P50)},
  MRNUMBER = {2719557},
MRREVIEWER = {Frank\ Neumann},
       DOI = {10.1515/CRELLE.2010.067},
       URL = {https://doi.org/10.1515/CRELLE.2010.067},
}

@book {platt,
    AUTHOR = {Platt, David},
     TITLE = {Chern character for global matrix factorizations},
      NOTE = {Thesis (Ph.D.)--University of Oregon},
 PUBLISHER = {ProQuest LLC, Ann Arbor, MI},
      YEAR = {2013},
     PAGES = {109},
      ISBN = {978-1303-28382-6},
   MRCLASS = {99-05},
  MRNUMBER = {3187371},
       URL =
              {http://gateway.proquest.com/openurl?url_ver=Z39.88-2004&rft_val_fmt=info:ofi/fmt:kev:mtx:dissertation&res_dat=xri:pqm&rft_dat=xri:pqdiss:3589551},
}

@misc{platt2012cherncharacterglobalmatrix,
      title={Chern Character for Global Matrix Factorizations}, 
      author={David Platt},
      year={2012},
      eprint={1209.5686},
      archivePrefix={arXiv},
      primaryClass={math.AG},
      url={https://arxiv.org/abs/1209.5686}, 
}

@article {yekutieli,
    AUTHOR = {Yekutieli, Amnon},
     TITLE = {The continuous {H}ochschild cochain complex of a scheme},
   JOURNAL = {Canad. J. Math.},
  FJOURNAL = {Canadian Journal of Mathematics. Journal Canadien de
              Math\'ematiques},
    VOLUME = {54},
      YEAR = {2002},
    NUMBER = {6},
     PAGES = {1319--1337},
      ISSN = {0008-414X,1496-4279},
   MRCLASS = {16E40 (14F10 18G10)},
  MRNUMBER = {1940241},
MRREVIEWER = {Peter\ J\o rgensen},
       DOI = {10.4153/CJM-2002-051-8},
       URL = {https://doi.org/10.4153/CJM-2002-051-8},
}

@article {caldararu-tu,
    AUTHOR = {C\u ald\u araru, Andrei and Tu, Junwu},
     TITLE = {Curved {$A_\infty$} algebras and {L}andau-{G}inzburg models},
   JOURNAL = {New York J. Math.},
  FJOURNAL = {New York Journal of Mathematics},
    VOLUME = {19},
      YEAR = {2013},
     PAGES = {305--342},
      ISSN = {1076-9803},
   MRCLASS = {18E30 (14F05)},
  MRNUMBER = {3084707},
MRREVIEWER = {Di-Ming\ Lu},
}

@article {PV-DUKE,
    AUTHOR = {Polishchuk, Alexander and Vaintrob, Arkady},
     TITLE = {Chern characters and {H}irzebruch-{R}iemann-{R}och formula for
              matrix factorizations},
   JOURNAL = {Duke Math. J.},
  FJOURNAL = {Duke Mathematical Journal},
    VOLUME = {161},
      YEAR = {2012},
    NUMBER = {10},
     PAGES = {1863--1926},
      ISSN = {0012-7094,1547-7398},
   MRCLASS = {14F05 (16E40 18G60)},
  MRNUMBER = {2954619},
MRREVIEWER = {David\ Favero},
       DOI = {10.1215/00127094-1645540},
       URL = {https://doi.org/10.1215/00127094-1645540},
}

@misc{chain-level-hkr-type-map-chern--kim,
      title={A chain-level HKR-type map and a Chern character formula}, 
      author={Kuerak Chung and Bumsig Kim and Taejung Kim},
      year={2021},
      eprint={2109.14372},
      archivePrefix={arXiv},
      primaryClass={math.AG},
      note={\url{https://arxiv.org/abs/2109.14372}}, 
}

@article {brown-walker-NG,
    AUTHOR = {Brown, Michael K. and Walker, Mark E.},
     TITLE = {A {C}hern-{W}eil formula for the {C}hern character of a
              perfect curved module},
   JOURNAL = {J. Noncommut. Geom.},
  FJOURNAL = {Journal of Noncommutative Geometry},
    VOLUME = {14},
      YEAR = {2020},
    NUMBER = {2},
     PAGES = {709--772},
      ISSN = {1661-6952,1661-6960},
   MRCLASS = {13D09 (14C35 16E40)},
  MRNUMBER = {4130844},
MRREVIEWER = {Jian\ Min\ Chen},
       DOI = {10.4171/jncg/378},
       URL = {https://doi.org/10.4171/jncg/378},
}

@article {ballard-favero-katzarkov-kernel,
    AUTHOR = {Ballard, Matthew and Favero, David and Katzarkov, Ludmil},
     TITLE = {A category of kernels for equivariant factorizations and its
              implications for {H}odge theory},
   JOURNAL = {Publ. Math. Inst. Hautes \'Etudes Sci.},
  FJOURNAL = {Publications Math\'ematiques. Institut de Hautes \'Etudes
              Scientifiques},
    VOLUME = {120},
      YEAR = {2014},
     PAGES = {1--111},
      ISSN = {0073-8301,1618-1913},
   MRCLASS = {14F05 (14C30 18E30 18G55)},
  MRNUMBER = {3270588},
MRREVIEWER = {Pawel\ Sosna},
       DOI = {10.1007/s10240-013-0059-9},
       URL = {https://doi.org/10.1007/s10240-013-0059-9},
}

@article {segal-closed-state,
    AUTHOR = {Segal, Ed},
     TITLE = {The closed state space of affine {L}andau-{G}inzburg
              {B}-models},
   JOURNAL = {J. Noncommut. Geom.},
  FJOURNAL = {Journal of Noncommutative Geometry},
    VOLUME = {7},
      YEAR = {2013},
    NUMBER = {3},
     PAGES = {857--883},
      ISSN = {1661-6952,1661-6960},
   MRCLASS = {14F43 (14F10)},
  MRNUMBER = {3108698},
MRREVIEWER = {Anatoly\ Libgober},
       DOI = {10.4171/JNCG/137},
       URL = {https://doi.org/10.4171/JNCG/137},
}

@article {vistoli-gw-th,
    AUTHOR = {Abramovich, Dan and Graber, Tom and Vistoli, Angelo},
     TITLE = {Gromov-{W}itten theory of {D}eligne-{M}umford stacks},
   JOURNAL = {Amer. J. Math.},
  FJOURNAL = {American Journal of Mathematics},
    VOLUME = {130},
      YEAR = {2008},
    NUMBER = {5},
     PAGES = {1337--1398},
      ISSN = {0002-9327,1080-6377},
   MRCLASS = {14N35 (14A20 53D45)},
  MRNUMBER = {2450211},
MRREVIEWER = {Johannes\ Walcher},
       DOI = {10.1353/ajm.0.0017},
       URL = {https://doi.org/10.1353/ajm.0.0017},
}

@article {hkr-original,
    AUTHOR = {Hochschild, G. and Kostant, Bertram and Rosenberg, Alex},
     TITLE = {Differential forms on regular affine algebras},
   JOURNAL = {Trans. Amer. Math. Soc.},
  FJOURNAL = {Transactions of the American Mathematical Society},
    VOLUME = {102},
      YEAR = {1962},
     PAGES = {383--408},
      ISSN = {0002-9947,1088-6850},
   MRCLASS = {18.20 (14.52)},
  MRNUMBER = {142598},
MRREVIEWER = {J.\ W.\ Gray},
       DOI = {10.2307/1993614},
       URL = {https://doi.org/10.2307/1993614},
}

@article {motivic-pippi,
    AUTHOR = {Pippi, Massimo},
     TITLE = {Motivic and {$\ell$}-adic realizations of the category of
              singularities of the zero locus of a global section of a
              vector bundle},
   JOURNAL = {Selecta Math. (N.S.)},
  FJOURNAL = {Selecta Mathematica. New Series},
    VOLUME = {28},
      YEAR = {2022},
    NUMBER = {2},
     PAGES = {Paper No. 33, 72},
      ISSN = {1022-1824,1420-9020},
   MRCLASS = {14F42 (14B05 14F08 19E08 32S30)},
  MRNUMBER = {4359568},
MRREVIEWER = {Satoshi\ Mochizuki},
       DOI = {10.1007/s00029-021-00734-2},
       URL = {https://doi.org/10.1007/s00029-021-00734-2},
}

@article {coherent-analogues,
    AUTHOR = {Efimov, Alexander I. and Positselski, Leonid},
     TITLE = {Coherent analogues of matrix factorizations and relative
              singularity categories},
   JOURNAL = {Algebra Number Theory},
  FJOURNAL = {Algebra \& Number Theory},
    VOLUME = {9},
      YEAR = {2015},
    NUMBER = {5},
     PAGES = {1159--1292},
      ISSN = {1937-0652,1944-7833},
   MRCLASS = {14F05 (13D09)},
  MRNUMBER = {3366002},
MRREVIEWER = {Scott\ R.\ Nollet},
       DOI = {10.2140/ant.2015.9.1159},
       URL = {https://doi.org/10.2140/ant.2015.9.1159},
}

@article{ArinkinCaldararuHablicsek2019,
  author  = {Arinkin, Dima and C\u{a}ld\u{a}raru, Andrei and Hablicsek, M\'{a}rton},
  title   = {Formality of derived intersections and the orbifold {HKR} isomorphism},
  journal = {Journal of Algebra},
  volume  = {540},
  pages   = {151--186},
  year    = {2019}
}

@article{FanJarvisRuan2013,
  author   = {Fan, Huijun and Jarvis, Tyler J. and Ruan, Yongbin},
  title    = {The {W}itten equation, mirror symmetry, and quantum singularity theory},
  journal  = {Annals of Mathematics},
  volume   = {178},
  number   = {1},
  pages    = {1--106},
  year     = {2013},
  mrnumber = {3043578}
}

@article{PolishchukVaintrob2016,
  author   = {Polishchuk, Alexander and Vaintrob, Arkady},
  title    = {Matrix factorizations and cohomological field theories},
  journal  = {Journal f{\"u}r die reine und angewandte Mathematik (Crelles Journal)},
  volume   = {714},
  pages    = {1--122},
  year     = {2016},
  mrnumber = {3491884}
}

@article{FaveroKim2020,
  author   = {Favero, David and Kim, Bumsig},
  title    = {General {GLSM} Invariants and Their Cohomological Field Theories},
  journal  = {arXiv preprint arXiv:2006.12182},
  year     = {2020}
}

@book{GLSMFundamental2022,
  author    = {Ciocan-Fontanine, Ionu\textcommabelow{t} and Favero, David and Gu\'{e}r\'{e}, J\'{e}r\'{e}my and Kim, Bumsig and Shoemaker, Mark},
  title     = {Fundamental Factorization of a {GLSM}, Part {I}: Construction},
  series    = {Memoirs of the American Mathematical Society},
  volume    = {280},
  number    = {1381},
  year      = {2022},
  publisher = {American Mathematical Society}
}

@misc{Halpern_Leistner_2020---arxiv--v2,
   title={Equivariant Hodge theory and noncommutative geometry},
   howpublished = {\url{https://arxiv.org/abs/1507.01924v2}},
   author={Halpern-Leistner, Daniel and Pomerleano, Daniel},
   year={2020}
}

\end{document}